\documentclass[english]{amsart}
\usepackage{a4wide}

\usepackage{comment}

\usepackage[backend=biber, citestyle=numeric-comp, bibstyle=trad-abbrv, url=false, doi=false, isbn=false]{biblatex}
\usepackage{amsmath}
\usepackage{mathtools}
\usepackage{amsfonts}
\usepackage{amsthm}
\usepackage{amssymb}
\usepackage{xcolor}

\usepackage{enumerate}

\usepackage{mathrsfs}
\usepackage{hyperref}
\usepackage{cleveref}

\newcommand{\R}{{\mathbb{R}}}

\newcommand{\N}{{\mathbb{N}}}

\newcommand{\vp}{{\varphi}}

\newcommand{\loc}{{\text{loc}}}

\def\d{{\,{\rm d}}}
\newcommand{\dx}{{\,{\rm d}x}}

\newcommand{\dt}{{\,{\rm d}t}}
\newcommand{\ds}{{\,{\rm d}s}}

\newcommand{\dmu}{{\,{\rm d}\mu}}
\newcommand{\dnu}{{\,{\rm d}\nu}}

\renewcommand{\Tilde}{\widetilde}
\newcommand{\coloneqq}{=}
\newcommand{\eqqcolon}{=}
\renewcommand{\div}{\operatorname{div}}

\newcommand{\cH}{{\mathcal{H}}}

\newcommand{\cA}{\mathcal{A}}

\newcommand{{\SP}}{{\mathcal{Z}}}

\DeclareMathOperator{\supp}{supp}

\DeclareMathOperator*{\osc}{osc}

\DeclareMathOperator{\dist}{dist}
\DeclareMathOperator{\diam}{diam}

\DeclareMathOperator*{\esssup}{ess\,sup}

\usepackage{refcount}

\newcounter{cte}

\definecolor{darkgreen}{rgb}{0.00, 0.50, 0.00}

\definecolor{mahogany}{rgb}{0.75, 0.25, 0.00}

\def\Xint#1{\mathchoice
   {\XXint\displaystyle\textstyle{#1}}%
   {\XXint\textstyle\scriptstyle{#1}}%
   {\XXint\scriptstyle\scriptscriptstyle{#1}}%
   {\XXint\scriptscriptstyle\scriptscriptstyle{#1}}%
   \!\int}
\def\XXint#1#2#3{{\setbox0=\hbox{$#1{#2#3}{\int}$}
     \vcenter{\hbox{$#2#3$}}\kern-.5\wd0}}

\def\dashint{\Xint-}
\def\Cz{{C}}

\def\YYint#1#2#3{{\setbox0=\hbox{$#1{#2#3}{\iint}$}
    \vcenter{\hbox{$#2#3$}}\kern-.51\wd0}}

\theoremstyle{plain}
\newtheorem{proposition}{Proposition}[section]
\newtheorem{lemma}[proposition]{Lemma}
\newtheorem{theorem}{Theorem}
\newtheorem{corollary}[proposition]{Corollary}

\theoremstyle{remark}
\newtheorem{remark}[proposition]{Remark}

\theoremstyle{definition}
\newtheorem{definition}{Definition}[section]

\crefname{theorem}{Theorem}{Theorems}
\Crefname{theorem}{Theorem}{Theorems}
\crefname{lemma}{Lemma}{Lemmas}
\Crefname{lemma}{Lemma}{Lemmas}
\crefname{proposition}{Proposition}{Propositions}
\Crefname{proposition}{Proposition}{Propositions}
\crefname{corollary}{Corollary}{Corollaries}
\Crefname{corollary}{Corollary}{Corollaries}
\crefname{definition}{Definition}{Definitions}
\Crefname{definition}{Definition}{Definitions}
\crefname{remark}{Remark}{Remarks}
\Crefname{remark}{Remark}{Remarks}
\crefname{example}{Example}{Examples}
\Crefname{example}{Example}{Examples}

\numberwithin{equation}{section}

\allowdisplaybreaks

\begin{document}

\title{Carleson-type removability for $p$-parabolic equations}

\author[M.~Borowski]{Michał Borowski}
\address{University of Warsaw, ul.~Banacha 2, 02-097 Warsaw, Poland}
\email{m.borowski@mimuw.edu.pl}

\author[T.~Elenius]{Theo Elenius}
\address{Department of Mathematics and Systems Analysis, Aalto University, P.O.~Box 11100, FI-00076 Aalto, Finland}
\email{theo.elenius@aalto.fi}

\author[L.~Schätzler]{Leah Schätzler}
\address{Department of Mathematics and Systems Analysis, Aalto University, P.O.~Box 11100, FI-00076 Aalto, Finland}
\email{ext-leah.schatzler@aalto.fi}

\author[D.~Stolnicki]{David Stolnicki}
\address{University of Warsaw, ul.~Banacha 2, 02-097 Warsaw, Poland}
\email{d.stolnicki-gonzalez@uw.edu.pl}




\begin{abstract}
We characterize removable sets for Hölder continuous solutions to degenerate parabolic equations of $p$-growth. A sufficient and necessary condition for a set to be removable is given in terms of an intrinsic parabolic Hausdorff measure, which depends on the considered Hölder exponent. We present a new method to prove the sufficient condition, which relies only on fundamental properties of the obstacle problem and supersolutions, and applies to a general class of operators. For the necessity of the condition, we establish the Hölder continuity of solutions with measure data, provided the measure satisfies a suitable decay property. The techniques developed in this article provide a new point of view even in the case $p=2$.
\end{abstract}

\makeatother
\maketitle

\section{Introduction}
This article concerns removable sets for continuous weak solutions of quasilinear parabolic equations of $p$-growth in the degenerate case $p \geq 2$. We consider the equation    \begin{equation}\label{eq_main}
    \partial_t u-\div \cA(x,t,\nabla u)=0\,,
\end{equation}
where the operator $\cA$ is monotone and of $p$-growth with respect to the last variable; for the precise assumptions, see Section \ref{sec:weak-solutions}.
The prototype is the parabolic $p$-Laplace equation
$$
    \partial_t u - \div \big( |\nabla u|^{p-2} \nabla u \big) = 0\,.
$$
For an open set $\Omega \subset \R^n$, $0<T<\infty$, and $\Omega_T = \Omega \times (0,T)$, the aim is to characterize those sets $E \subset \Omega_T$ such that solutions to \eqref{eq_main} in $\Omega_T \setminus E$, with prescribed modulus of continuity up to \(E\), are in fact solutions in the whole of $\Omega_T$.
In particular, we are interested in removability in a class of Hölder continuous functions, under a condition on the size of the set $E$ in terms of a parabolic Hausdorff measure. In such a setting, sets $E$ of nonzero capacity can be removed.

To state our results, we recall the intrinsic scaling method for equations of type \eqref{eq_main}. Using this technique, an extensive regularity theory has been developed.
General references in this connection are the monographs \cite{DiBenedettoGianazzaVespri2012,DiBenedetto1993, Urbano}. The idea is to compensate for the lack of scaling in equation \eqref{eq_main} by using an intrinsic geometry. This is done by considering cylinders of the type 
\begin{equation*}
    Q_{r,\theta}(x_0,t_0) = B(x_0,r)\times\big(t_0-r^p\theta^{2-p},t_0+r^p\theta^{2-p}\big)\,,
\end{equation*}
where the parameter $\theta$ controls the size of the solution $u$ in the same cylinder. Oscillation estimates and Harnack inequalities for solutions hold in such intrinsic geometries related to the operator and the solution itself. As we are considering Hölder continuous functions, the oscillation of $u$ is of order $\approx\lambda r^\alpha$, where the parameter $\lambda$ corresponds to the Hölder coefficient of the solution in the intrinsic cylinder and $0<\alpha\leq 1$. As such, we are led to consider cylinders of the form
\begin{equation*}
    Q_{r,\lambda r^\alpha}(x_0,t_0)= B(x_0,r)\times\big(t_0-r^{\,p+\alpha(2-p)}\lambda^{2-p},\,t_0+r^{\,p+\alpha(2-p)}\lambda^{2-p}\big)\,.
\end{equation*}
We define the corresponding metric related to these cylinders by
\[
d_\alpha\big((x,t),(y,s)\big) \coloneqq \max\Big\{|x-y|,\ |t-s|^{\frac{1}{\,p+\alpha(2-p)\,}}\Big\},\qquad 0<\alpha\le1\,.
\]
Note that the exponent $p+\alpha(2-p)$ is a convex combination of $2$ and $p$.
For further discussion on these interpolative intrinsic geometries, we refer to~\cite{kuusi2014sharp}. The related Hausdorff measure is given by
\begin{equation}
    \label{eq: def Hausdorff measure}
    \cH^{\sigma}_{(\alpha)}(E)
= \lim_{\delta\downarrow 0}\ \inf\left\{\sum_{i=1}^\infty r_i^{\sigma}:\ 
E\subset \bigcup_{i=1}^\infty Q_{r_i,r_i^\alpha}(z_i),\ 0<r_i<\delta\right\},    
\end{equation}
for a set $E\subset\R^{n+1}$. In the case $\sigma=n$ and $\alpha=0$, the measure~\eqref{eq: def Hausdorff measure} reduces to the Hausdorff measure related to the $p$-parabolic capacity, see \cite{MoringScheven}.

We are now ready to state our first main result, which shows that the condition $\cH^{n+\alpha}_{(\alpha)}(E) = 0$ is sufficient for the removability of the set $E$ in the class of $\alpha$-Hölder continuous solutions. More precisely, we prove the following result.
\begin{theorem}
\label{thm:main}
    Let $E\subset\Omega_T$ be a closed set, let $\cA$ satisfy \eqref{as:growth} and \eqref{as:monotonicity-weaker}, and assume that $u\in \Cz(\Omega_T)$ is a weak solution of \eqref{eq_main} in $\Omega_T\backslash E$ such that 
    \begin{equation}
        \left| u(z)-u(z_0) \right|
        \leq
        Md^\alpha_\alpha(z,z_0)\,,
        \label{eq:Holder-condition-E}
    \end{equation}
    holds for every $z_0 \in E$ and $z \in \Omega_T$. If $\mathcal{H}^{n+\alpha}_{(\alpha)}(E) = 0$, then $u$ is a weak solution to \eqref{eq_main} in $\Omega_T$.
\end{theorem}

The converse result to Theorem~\ref{thm:main} is also true under slightly stronger assumptions. First of all, we require strict monotonicity of the operator $\cA$ in Theorem~\ref{thm: necessity of Hausdorff zero}. Nonetheless, a broad class of operators is considered, including the $p$-Laplacian. The upper threshold $\kappa$ appearing in~\Cref{thm: necessity of Hausdorff zero} is related to the optimal local H\"older continuity of solutions to~\eqref{eq_main}. For smooth operators $\cA$, one can always take $\kappa = 1$, see \cite{KuusiMingione2013}.
To the best of our knowledge, Theorem~\ref{thm: necessity of Hausdorff zero} is new even in the case $p=2$.
\begin{theorem}
\label{thm: necessity of Hausdorff zero}
Let $\cA$ satisfy \eqref{as:growth} and \eqref{as:monotonicity-stronger}, let $\kappa \in (0, 1]$ be such that~\eqref{eq: kappa hölder cont of sols} holds, and let $0<\alpha<\kappa$. Let $E\subset\Omega_T$ be a closed set with $\mathcal{H}^{n+\alpha}_{(\alpha)}(E)>0$. Then, there exists a function $u$ which is locally $\alpha$-Hölder continuous in $\Omega_T$ with respect to $d_\alpha$, such that $u$ is a solution of \eqref{eq_main} in $\Omega_T\backslash E$, but not in $\Omega_T$.  
\end{theorem}

The removability of singularities of H\"older continuous solutions to PDEs was first treated in the case of harmonic functions in the seminal paper~\cite{Carleson} by Carleson.
For general elliptic equations of $p$-Laplace type, Kilpeläinen and Zhong \cite{KilpelainenZhongRemovability} provide both sufficient and necessary conditions for a set to be removable.
In particular, in the case of $\alpha$-H\"older continuous solutions, $0<\alpha<1$, the authors give a full characterization of removable sets for the prototype $p$-Laplacian.
Indeed, they conclude that in this case, $E$ is removable if and only if its $(n-p+\alpha(p-1))$-dimensional Hausdorff measure is zero.
The case of Lipschitz continuous solutions is essentially different, see~\cite{DavidMattila, NazarovTolsaVolberg,mateuPratTolsa}.
In the elliptic setting, generalizations in several directions are available, see \cite{Ono} for problems including lower-order terms, and \cite{Iwonka1, Iwonka2} for operators with Orlicz growth.

In the parabolic setting, Kuusi, Mingione and Nyström \cite{kuusi2014sharp} consider autonomous equations of $p$-Laplacian type given by $\partial_t u-\cA(\nabla u)=0$, with $\cA\in C^1$.
For solutions that are continuous with a modulus of continuity $\omega(\cdot)$, the authors prove that sets with $\mathcal{H}^{\omega(\cdot)}(E)=0$ are removable, where the Hausdorff measure $\mathcal{H}^{\omega(\cdot)}$ is a generalization of \eqref{eq: def Hausdorff measure}.
In the Hölder case $\omega(r) = c r^\alpha$, the condition $\mathcal{H}^{n+\alpha}_{(\alpha)}(E) = 0$ in Theorem \ref{thm:main} is retrieved.

At this stage, let us indicate that our results are natural parabolic counterparts of the elliptic results in \cite{KilpelainenZhongRemovability}.
Theorem~\ref{thm:main} is analogous to \cite[Theorem 1.6]{KilpelainenZhongRemovability}, and Theorem~\ref{thm: necessity of Hausdorff zero} corresponds to~\cite[Theorem 1.10]{KilpelainenZhongRemovability}.
Thus, let us compare our proof strategy to the elliptic and parabolic benchmarks \cite{KilpelainenZhongRemovability} and \cite{kuusi2014sharp}.

We start with the sufficient condition concerning the size of the set $E$.
In \cite{KilpelainenZhongRemovability}, the overall strategy is the following:
Given a solution $u$ outside the set $E \subset \Omega$, consider a solution $v$ to the obstacle problem in the whole set $\Omega$, with obstacle function $u$. By Harnack's inequality, the oscillation of $v$ can be controlled via the oscillation of $u$ near the contact set. In turn, Caccioppoli estimates yield an upper bound for the Riesz measure $\mu$ of $v$, which, in conjunction with the assumption on the size of $E$, implies that $\mu(E) = 0$. On the other hand, the measure $\mu$ of the complement $\Omega \setminus E$ is also zero, as $u$ is a solution there. This leads to the conclusion that $\mu \equiv 0$ in $\Omega$, which guarantees that $v$ is a solution to the PDE.
Finally, establishing that $v$ coincides with $u$ shows that $u$ is a solution in the whole set $\Omega$. 

In the parabolic setting, using the Harnack inequality is more delicate due to the presence of the time lag. In \cite{kuusi2014sharp}, the stronger assumption on the operator allows reducing to the case $p=2$ when proving the oscillation estimate for the obstacle problem. In this case, the time lag in the Harnack inequality can be controlled up to a constant factor. This eventually gives the desired estimate of the oscillation of the solution $v$ of the obstacle problem in terms of the oscillation of the obstacle $u$.
In our setting, due to the generality of the operators in \eqref{eq_main}, the regularity of solutions is not sufficient to perform the reduction to $p=2$. Moreover, the time lag cannot be controlled by a constant scaling in the intrinsic Harnack inequality for degenerate equations, that is, for $p>2$. As such, proving the oscillation estimate in our setting seems to be a difficult problem of independent interest.

In \Cref{sec:sufficiency}, we develop an argument which directly shows the decay of the Riesz measure associated with the solution $v$ of the obstacle problem, while avoiding the oscillation estimate for $v$ itself. Our approach is based only on fundamental tools, in particular, a Caccioppoli inequality, truncation arguments, and elementary properties of the obstacle problem. The proof proceeds as follows. First, for any supersolution $w$ of \eqref{eq_main}, we estimate the Riesz measure associated with $w$ in terms of the oscillation of $w$. Next, by truncating $v$, the oscillation of the truncated function is controlled by the oscillation of the obstacle. Due to properties of the obstacle problem, this truncation only increases the associated Riesz measure, which in turn implies that the Riesz measure of $v$ can be estimated by the oscillation of the obstacle.

In \Cref{sec:necessity}, we discuss the necessity of the condition on the size of the singular set $E$.
The proof of \Cref{thm: necessity of Hausdorff zero} is based on constructing a nonzero measure $\mu$ supported on $E$, and a function $u$ satisfying 
\[
\partial_t u-\cA(x,t,\nabla u)=\mu \,,
\]
in the weak sense. The construction of the measure $\mu$ uses a parabolic Frostman Lemma, which also yields a suitable decay for the measure. The decay of the measure, in turn, ensures the H\"older continuity of the solution to the measure data problem. This is done by applying a suitable comparison estimate for solutions to measure data problems and a delicate intrinsic iteration technique. 

\section{Assumptions and preliminary results}
\label{sec:assumptions-preliminaries}
In this section, we give the precise setting and gather notions and results used throughout the paper. We start by introducing some notation.
\subsection{Notation}
For $0<t_1<t_2<\infty$ we denote 
\begin{equation*}
     \Omega_{t_1,t_2}\coloneqq \Omega\times(t_1,t_2)\,.
\end{equation*}
The parabolic boundary of $\Omega_{t_1,t_2}$ is 
\begin{equation*}
    \partial_\mathrm{par} \Omega_{t_1,t_2} = (\partial \Omega \times [t_1,t_2))\cup (\overline{\Omega}\times\{t_1\})\,.
\,\end{equation*}
We shall use the notation $\Omega_t$ for $\Omega \times [0, t)$. For $z_0=(x_0,t_0)\in\mathbb{R}^{n+1}$, $r>0$, and $\lambda>0$, intrinsic cylinders are defined as
\[
Q_{r,\lambda}(z_0)=B(x_0,r)\times\big(t_0-\lambda^{2-p}r^p,\; t_0+\lambda^{2-p}r^p\big)\,.
\]
If it is clear from the context, we omit $z_0$.
For any real-valued function $v$ defined on a set $U$, we shall denote
\begin{equation*}
    \osc_U v \coloneqq \sup_{U} v - \inf_{U} v\,.
\end{equation*}
Moreover, we shall use the notation $v_U$ the mean-value of $v$ on $U$, that is,
\begin{equation*}
    v_U \coloneqq \dashint_U v(x)\dx\,.
\end{equation*}
Throughout the paper, letters $C$ and $c$ denote positive constants that may differ from line to line. We shall indicate on what variables those constants depend on their first appearance within a given theorem or lemma.
\subsection{Structural assumptions}
\label{sec:weak-solutions}
We assume that $\cA:\Omega_T\times\mathbb{R}^n\to\mathbb{R}^n$ is a Carathéodory function, that is, $(x,t)\mapsto\cA(x,t,\xi)$ is measurable for every $\xi\in\R^n$ and $\xi\mapsto\cA(x,t,\xi)$ is continuous for a.e. \ $(x,t)\in\Omega_T$. For $p\geq 2$, we assume that for a.e.~$(x, t) \in \Omega_T$ and all $\xi \in \R^n$, $\cA$ satisfies the structural conditions 
\begin{equation}\tag{S1}\label{as:growth}
        \cA(x,t,\xi) \cdot \xi \geq \nu |\xi|^p
        \quad \text{and} \quad
        |\cA(x,t,\xi)| \leq L |\xi|^{p-1}
\end{equation}
with positive constants $0 < \nu \leq L$.
Moreover, we assume that $\mathcal{A}$ satisfies one of the following monotonicity assumptions
        \begin{equation}\tag{S2a}\label{as:monotonicity-weaker}
            \left(\cA(x,t,\xi)-\cA(x,t,\eta) \right) \cdot \left( \xi - \eta \right) > 0 \text{ for all $\xi$, $\eta \in \R^n$ with } \xi \neq \eta\,,\end{equation} 
        \begin{equation}\tag{S2b}\label{as:monotonicity-stronger}
            \left(\cA(x,t,\xi)-\cA(x,t,\eta) \right) \cdot \left( \xi - \eta \right) \geq \nu \lvert\xi-\eta\rvert^p \text{ for all $\xi$,$\eta\in \R^n\,,$}
        \end{equation}
for a.e.~$(x, t) \in \Omega_T$.
By default, we shall assume that $\cA$ satisfies~\eqref{as:monotonicity-weaker}, while we shall always indicate whenever imposing~the stronger conditon~\eqref{as:monotonicity-stronger}.

\subsection{Weak sub- and supersolutions}
\label{sec:weak sub- and supersolutions}
The parabolic Sobolev space $L^p( t_1,t_2;W^{1,p}( \Omega))$ consists of measurable functions $u:\Omega_{t_1,t_2}\to [-\infty,\infty]$, such that for almost every $t\in (t_1,t_2)$, we have $u(\cdot, t)\in W^{1,p}(\Omega)$ and 
\begin{equation*}
    \int_{t_1}^{t_2} \lVert u(\cdot, t) \rVert_{W^{1,p}( {\Omega})}^p \dt < \infty\,.
\end{equation*}
We say that $u \in L^p_{\loc}(t_1, t_2;W_{\loc}^{1, p}(\Omega))$, if for every $\Tilde \Omega_{s_1, s_2} \Subset \Omega_{t_1,t_2}$, we have $u \in L^p( s_1, s_2; W^{1, p}(\Tilde \Omega))$.

Now, we give the definition of weak sub- and supersolutions.
\begin{definition}\label{weak_sol}
Let $U$ be an open subset of $\R^{n+1}$. We say that $u$ is a weak solution to~\eqref{eq_main} in $U$ if for every $\Omega_{t_1, t_2} \Subset U$, we have $u \in L^p(t_1, t_2;W^{1, p}(\Omega))$, and for every $\varphi\in C_c^\infty(\Omega_{t_1, t_2})$, it holds
\begin{equation}
\label{eq: def of weak sol}
\int_{\Omega_{t_1,t_2}}\! -u\,\partial_t\varphi + \cA(x,t,\nabla u)\cdot \nabla\varphi \,\dx\dt=0\,.
\end{equation}
The function $u$ is a weak supersolution if the integral in \eqref{eq: def of weak sol} is $\ge 0$ for all non-negative $\varphi$, and a weak subsolution if the integral in \eqref{eq: def of weak sol} is nonpositive for all non-negative $\varphi$.
\end{definition}

We will need a basic existence result for the Cauchy-Dirichlet problem related to \eqref{eq_main} with continuous initial and boundary values in cylinders $Q$. For the following lemma, we refer to \cite[Lemma 3.2]{KL1996}. Similar statements have been used in \cite{KuusiMingione2014,kuusi2014sharp,KKS-Obstacle}.
\begin{lemma}
\label{lem:existence-Cauchy-Dirichlet}
Let $Q = B(z,r) \times (t_1,t_2) \subset \R^{n+1}$ be a cylinder, and assume that $g \in \Cz(\overline{Q})$.
Then, there exists a unique weak solution $u$ to \eqref{eq_main} in the sense of Definition \ref{weak_sol} with $u=g$ on $\partial_\mathrm{par}Q$.
Moreover, we have that $u \in \Cz(\overline{Q})$.
\end{lemma}
We have the following standard comparison principle, see \cite[Lemma 3.5]{kortekuusiparviainen2010-superparabolic}.
\begin{lemma}
\label{lem:standard-comparison}
Let $\Omega_{t_1,t_2} \subset \R^{n+1}$ be open, and assume that $u$ is a weak supersolution and $v$ a weak subsolution to~\eqref{eq_main} in $\Omega_{t_1,t_2}$, in the sense of Definition \ref{weak_sol}.
If $u$ and $-v$ are lower semicontinuous on $\overline{\Omega_{t_1,t_2}}$, and $v \leq u$ on $\partial_\mathrm{par}\Omega_{t_1,t_2}$, then $v \leq u$ a.e.~in $Q$.
\end{lemma}
Further, we give the following comparison principle applicable in noncylindrical domains, see \cite[Lemma 4.3]{KL1996} and \cite[Corollary 4.6]{kortekuusiparviainen2010-superparabolic}.
\begin{lemma}
\label{lem:elliptic-comparison}
Let $0<T<\infty$, and let $U \subset \R^n \times (-\infty,T)$ be open and bounded.
Further, assume that $u,v \in \Cz(\overline{U})$, and that $u$ is a weak supersolution and $v$ is a weak subsolution to \eqref{eq_main} in $U$, in the sense of Definition \ref{weak_sol}.
If $v \leq u$ on $\partial U \setminus \{t=T\}$, then $v \leq u$ in $U$.
\end{lemma}

Next, we state a Caccioppoli inequality for weak supersolutions, see~\cite[Lemma 5.2]{kortekuusiparviainen2010-superparabolic}.
\begin{lemma} \label{lem:Cacciopoli}
Let \(\Omega \subset \R^n\) be open, and let \(u\) be a weak supersolution to~\eqref{eq_main} in $\Omega_T$ such that \(-M \le u \le M\).
Then there exists a constant \(C=C(p,\nu,L)>0\) such that
\[
\int_{\Omega_T} \eta^{p} |\nabla u|^{p}\dx\dt
\le
C M^{2}\int_{\Omega_T}\left|\partial_t \eta^{p}\right|\dx\dt
\;+\;
C M^{p}\int_{\Omega_T} |\nabla \eta|^{p}\dx\dt\,,
\]
whenever \(\eta \in \Cz^{\infty}_0(\Omega_T)\) and \(\eta \ge 0\).
\end{lemma}

For sequences of uniformly bounded supersolutions, we have the following result, see \cite[Theorem 5.3, Corollary 5.7]{kortekuusiparviainen2010-superparabolic}.
\begin{lemma}
\label{lem:sequences-supersolutions}
Let $U \subset \R^{n+1}$ be open, and let $(u_i)_{i \in \N}$ be a sequence of weak supersolutions to \eqref{eq_main} in $U$.
Assume that there exists $M>0$ and a limit map $u \colon U \to \R$ such that $|u_i| \leq M$ for all $i \in \N$ and $u_i \to u$ a.e.~as $i \to \infty$.
Then, $u$ is a weak supersolution to \eqref{eq_main} in $U$, and for any $1 \leq q < p$, we have that $\nabla u_i \to \nabla u$ in $L^q_\mathrm{loc}(U)$ as $i \to \infty$.
\end{lemma}

The following supremum estimate for subsolutions is taken from~\cite[Thm 3.4]{Kuusi2009}.
\begin{lemma}
\label{lemma: sup estimate for subsolutions}
Let $U$ be an open subset of $\mathbb{R}^{n+1}$.
Suppose that $u$ is a non-negative subsolution to \eqref{eq_main} in $U$ and $B_R(x_0) \times (t_0 - T, t_0) \subset U$. Then there exists a constant $C = C(n, p, \nu,L)$ such that for any $0<\sigma<1$, we have
\[
\esssup_{B(x_0,\sigma R)\times(t_0 - \sigma^pT, t_0)} u 
\leq  \left( \frac{ R^p}{T} \right)^{1/(p-2)}
+ \frac{C}{(1-\sigma)^{n+p}}\frac{T}{R^p} \dashint_{B(x_0,R)\times(t_0-T,t_0)} u^{p-1} \dx \dt\,.
\]
\end{lemma}
In \Cref{sec:necessity}, we will need the following result on expansion of positivity, cf.~\cite[Proposition 4.1,p.~66]{DiBenedettoGianazzaVespri2012}. See also \cite[Lemma 3.1]{DiBenedettoGianazzaVespri2008} and \cite[Proposition 3.1]{Kuusi2008}.
Note that $C_1$ can be chosen larger by decreasing $\eta$.
\begin{lemma}
\label{lemma: expansion of pos}
Let $u$ be a non-negative weak supersolution to \eqref{eq_main} in $B(x_0,4R)\times(t_0,t_0+T_0)$. Let $t_0<s<t_0+T_0$, $k>0$ and suppose that
\begin{equation*}
    \big\lvert \big\{x\in B(x_0,R): u(x,s) \geq k\big\} \big\rvert \geq \frac{1}{2} \lvert B(x_0,R)\rvert. 
\end{equation*}
There exists constants $C_1>1$ and $\eta>0$ depending only on $(n,p,\nu,L)$ such that if $s+C_1 k^{2-p} R^p<t_0+T_0$, then
\begin{equation*}
    \inf_Q u\geq \eta k, 
\end{equation*}
where $Q=B(x_0,2R)\times \left( s+\frac{1}{2}C_1 k^{2-p}R^p,s+C_1 k^{2-p}R^p \right)$.
\end{lemma}

\subsection{Measure data problems and superparabolic functions} \label{sec:supersolutions-measures}
In this section, we discuss solutions to measure data problems. For simplicity, we consider only cylindrical domains. Let $\mu$ be a Radon measure on $\Omega_T$. We consider the equation with right-hand side 
\begin{equation}
    \label{eq: measure RHS}
    \partial_t u - \div \mathcal{A}(x,t,\nabla u)=\mu\,.
\end{equation}

\begin{definition}
\label{def:solution-measure-data}
A function $u\in L^{p-1}_{\loc}\big( 0,T;W^{1,p-1}_{\loc}(\Omega) \big)$ is called a weak solution of \eqref{eq: measure RHS} if 
\begin{equation*}
    \int_{\Omega_{T}}-u\,\partial_t\varphi + \cA(x,t,\nabla u)\cdot \nabla\varphi \dx\dt = \int_{\Omega_T} \varphi \dmu\,,
\end{equation*}
for all $\varphi\in \Cz^\infty_c(\Omega_T)$.
\end{definition}
Note that a function $u$ in the space $L^{p}_{\loc}\big(0,T;W^{1,p}_{\loc}(\Omega) \big)$ is a weak supersolution in the sense of Definition \ref{weak_sol} if and only if it is a weak solution in the sense of Definition \ref{def:solution-measure-data} with some Radon measure $\mu$.
Indeed, on the one hand, let $u$ be a supersolution in $\Omega_T$ in the sense of \Cref{weak_sol}.
Since
$$
    C^\infty_0(\Omega_T) \ni \varphi \mapsto \int_{\Omega_T}-u\,\partial_t\varphi + \cA(x,t,\nabla u)\cdot \nabla\varphi \dx\dt
$$
defines a non-negative linear functional, the Riesz Representation Theorem gives a measure $\mu$ such that $u$ is a weak solution of \eqref{eq: measure RHS}.
On the other hand, if $u\in L^{p}_{\loc}\big(0,T;W^{1,p}_{\loc}(\Omega) \big)$ is a weak solution of \eqref{eq: measure RHS}, then $u$ is a weak supersolution in the sense of Definition \ref{weak_sol}, since $\mu$ is non-negative.

However, given an arbitrary Radon measure $\mu$, the corresponding solution in the sense of Definition \ref{def:solution-measure-data} is in general only contained in the larger space $u\in L^{p-1}_{\loc}(0,T;W^{1,p-1}_{\loc}(\Omega))$. Therefore, following \cite{KinnunenLukkariPariviainen2010}, we give the following definition of superparabolic functions; see also \cite{KinnunenLindqvist2006}.
\begin{definition}
A function $u \colon \Omega_T\to (-\infty,\infty]$ is $\mathcal{A}$-superparabolic in $\Omega_T$ if it satisfies the following conditions:
\begin{enumerate}[i)]
    \item $u$ is lower semicontinuous;
    \item $u$ is finite in a dense subset of $\Omega_T$;
    \item if $h$ is a solution to~\eqref{eq_main} in $U_{t_1,t_2} \Subset\Omega_T$, continuous in $\overline{U_{t_1,t_2}}$, and $h\leq u$ on the parabolic boundary of $U_{t_1,t_2}$, then $h\leq u$ in $U_{t_1,t_2}$. 
\end{enumerate}
\end{definition}
By the results of \cite{KinnunenLukkariPariviainen2010}, if $u$ is superparabolic in $\Omega_T$, then $u\in L^{p-1}_{\loc}(0,T;W^{1,p-1}_{\loc}(\Omega))$ and there exists a Radon measure $\mu$ such that $u$ is a weak solution of \eqref{eq: measure RHS}. Moreover, for any $k\in \mathbb R$ we have $\min(u,k)\in L^{p}_{\loc}(0,T;W^{1,p}_{\loc}(\Omega))$ and $\min(u,k)$ is a weak supersolution of \eqref{eq_main}.
The associated Riesz measures $\mu_k$ converge weakly to $\mu$, since by the Dominated Convergence Theorem, we have that
\begin{equation}
\begin{split}
    \lim_{k \to \infty} \int_{\Omega_T} \varphi \d\mu_k
    &=
    \lim_{k \to \infty} \int_{\Omega_T} -\min(u,k) \partial_t \varphi + \cA(x,t,\nabla \min(u,k)) \cdot \nabla \varphi \dx\dt \\
    &=
    \int_{\Omega_T} -u \partial_t \varphi + \cA(x,t,\nabla u) \cdot \nabla \varphi \dx\dt
    =
    \int_{\Omega_T} \varphi \d\mu
    \label{eq:weak-convergence-measures-truncations}
\end{split}
\end{equation}
for any $\varphi \in C^\infty_c(\Omega_T)$.
We recall the following results regarding superparabolic functions, see \cite[Theorem 5.8, Theorem 5.12]{KinnunenLukkariPariviainen2010}.
\begin{lemma}\label{lemma:exis_measure}
Let $\mu$ be a non-negative finite Radon measure in $\Omega_T$. Then there exists an $\cA$-superparabolic function $u$ in $\Omega_T$ such that $u$ is a weak solution to~\eqref{eq: measure RHS}. Moreover, $u$ is a weak solution to~\eqref{eq_main} in $\Omega_T\setminus \operatorname{supp}\mu$, in the sense of \Cref{weak_sol}.
\end{lemma}

\subsection{The obstacle problem}
Let $\Omega_{t_1,t_2} \subset \R^{n+1}$ be open and bounded.
The \emph{obstacle problem} associated with \eqref{eq_main} in $\Omega_{t_1,t_2}$, with obstacle function $\psi$ is given by
\begin{equation}
    \left\{
    \begin{array}{ll}
        \partial_t u - \div \cA(x,t,\nabla u) = 0 & \text{in } \Omega_{t_1,t_2}\,, \\[5pt]
        u=\psi & \text{on } \partial_\mathrm{par}\Omega_{t_1,t_2}\,, \\[5pt]
        u \geq \psi & \text{in } \Omega_{t_1,t_2}\,.
    \end{array}
    \right.
    \label{eq:obstacle-problem}
\end{equation}
The precise definition of solution is the following, see \cite[Definition 2]{kuusi2014sharp}.
\begin{definition}
\label{def:obstacle-problem}
Let $ \psi \colon \overline{\Omega_{t_1,t_2}} \to \R$ be continuous functions that satisfy the compatibility condition $g \geq \psi$ on $\partial_\mathrm{par}\Omega_{t_1,t_2}$.
A continuous function $u \colon \overline{\Omega_{t_1,t_2}} \to \R$ is a solution to \eqref{eq:obstacle-problem} if it satisfies the following properties:
\begin{enumerate}[i)]
    \item $u \geq \psi$ in $\Omega_{t_1,t_2}$ and $u=g$ on $\partial_\mathrm{par}\Omega_{t_1,t_2}$;
    \item $u$ is a weak supersolution to \eqref{eq_main} in $\Omega_{t_1,t_2}$ in the sense of Definition \ref{weak_sol};
    \item $u$ is a weak solution to \eqref{eq_main} in $\Omega_{t_1,t_2} \cap \{u > \psi\}$ in the sense of Definition \ref{weak_sol};
    \item $u$ is the smallest weak supersolution above $\psi$, that is, if $v$ is a weak supersolution in $\Omega_{t_1,t_2}$ and $v\geq \psi$, then $v\geq u$.
\end{enumerate}
Given $\psi$, we denote the solution to the respective obstacle problem by $u = R^\psi = R^\psi_\cA$, omitting the subscript $\cA$ if it is unambiguous from the context.
\end{definition}
From \cite[Theorem 3.1]{KKS-Obstacle}, we infer the following existence result.
Note that \cite{KKS-Obstacle} deals with more general regularity conditions on $\Omega_{t_1,t_2}$.
\begin{lemma}
\label{lem:obstacle-existence}
Let $\Omega_{t_1,t_2} \subset \R^{n+1}$ be a space-time cylinder with smooth lateral boundary, and assume that $\psi \colon \overline{\Omega_{t_1,t_2}} \to \R$ is continuous.
Then, there exists a unique solution $u = R^\psi$ to the obstacle problem \eqref{eq:obstacle-problem} in $\Omega_{t_1,t_2}$ in the sense of Definition \ref{def:obstacle-problem}.
\end{lemma}

\section{Proof of Theorem \ref{thm:main}}\label{sec:sufficiency}
The proof of Theorem~\ref{thm:main} is composed of several ingredients, presented mainly in Lemmas~\ref{lemma: general estimate for measure of supersol}-\ref{lemma: decay of measure}. A crucial step in the proof is to establish a decay estimate for the Riesz measure associated to the solution of the obstacle problem with H\"older continuous obstacle, see Lemma~\ref{lemma: decay of measure}. To this end, we provide estimates concerning supersolutions to~\eqref{eq_main}, expressed in Lemmas~\ref{lemma: general estimate for measure of supersol} and~\ref{lem:meas-inq}. We begin with the following estimate.
{
\begin{lemma}
\label{lemma: general estimate for measure of supersol}
    Let $\Omega \subset \mathbb{R}^n$ be open. Let $u$ be a non-negative and locally bounded supersolution of \eqref{eq_main} in $\Omega_T$ and let 
    \begin{equation}\label{eq:measure_prob}
        \mu = \partial_t u  - \div \cA (x,t,\nabla u) ,
    \end{equation}
    be the associated Riesz measure.
    Let $z_0 \in \Omega_T$ and $r > 0$ be such that $Q \coloneqq Q_{2r, (2r)^{\alpha}}(z_0) \Subset \Omega_T$.
    Then, there exists a positive constant $C = C(n, p, \nu, L)$ such that
    \begin{equation*}
        \mu \left(Q_{r, r^{\alpha}}(z_0) \right) \leq Cr^n \left(r^{(2-p)\alpha}\sup_Q u^{p-1} + r^{\frac{(2-p)\alpha}{p}}\sup_Q u^{\frac{2(p-1)}{p}} + \sup_Q u\right)\,.
    \end{equation*}
\end{lemma}
\begin{proof}
Let $\eta$ be a standard cut-off function such that $0 \leq \eta \leq 1$, $\eta \equiv 1$ on $Q_{r, r^{\alpha}}(z_0)$, $\supp \eta \subseteq Q$, and $|\nabla \eta| \leq \frac{c}{r}$, $|\partial_t \eta| \leq \frac{c}{r^{p + (2-p)\alpha}}$ for a constant $c = c(n)$. Then, by the definition of $u$ and Assumption \eqref{as:growth} on $\cA$, we have
    \begin{equation}\label{eq:sep25-1}
    \begin{split}
        \mu \left(Q_{r, r^{\alpha}}(z_0) \right) \leq \int_{Q} \eta^p d\mu &= \int_{Q} p\eta^{p-1} \cA(x, t, \nabla u) \cdot \nabla \eta \dx \dt - \int_{Q} pu\eta^{p-1} \partial_t \eta \dx \dt\\
        &\leq C\int_Q  \eta^{p-1}|\nabla u|^{p-1}|\nabla \eta| + u|\partial_t \eta|  \dx \dt\,,
    \end{split}
    \end{equation}
    where the first inequality holds since $u$ is a weak solution of \eqref{eq:measure_prob} in the sense of Definition \ref{def:solution-measure-data}. Furthermore, by the bound on $\partial_t\eta$, we have
    \begin{equation}\label{eq:sep25-2}
        \int_{Q} u|\partial_t \eta|\dx\dt \leq Cr^n \sup_Q u\,.
    \end{equation}
    On the other hand, H\"older's inequality together with Lemma~\ref{lem:Cacciopoli} gives
    \begin{align*}
        \int_Q \eta^{p-1}|\nabla u|^{p-1}|\nabla \eta|\dx\dt &\leq
        \left( \int_Q \eta^p |\nabla u|^p \dx \dt \right)^{\frac{p-1}{p}}\left( \int_Q |\nabla \eta|^p \dx \dt \right)^{\frac{1}{p}}\\
        &\leq Cr^{\frac{n + (2-p)\alpha}{p}} \left(\sup_Q u^{p} \int_{Q} |\nabla \eta|^p\dx \dt + \sup_Q u^{2}\int_Q |\partial_t \eta^p| \dx \dt\right)^{\frac{p-1}{p}}\\
        &\leq Cr^{n+ (2-p)\alpha}\sup_Q u^{p-1} + Cr^{n + \frac{(2-p)\alpha}{p}}\sup_Q u^{\frac{2(p-1)}{p}}\,.
    \end{align*}
    Combining the last estimate with~\eqref{eq:sep25-1} and~\eqref{eq:sep25-2} yields the desired conclusion.
\end{proof}
}
Truncations of supersolutions to~\eqref{eq_main} are themselves supersolutions, see \cite[Lemma 3.2]{kortekuusiparviainen2010-superparabolic} for the proof in our setting.
Tracing the two Riesz measures in this proof produces the estimate in the following lemma. For the convenience of the reader, we give the full proof.
\begin{lemma}\label{lem:meas-inq}
Let $\Omega \subset \mathbb{R}^n$ be open. Let $u$ be a continuous supersolution to \eqref{eq_main} in $\Omega_T$ and let 
\begin{equation*}
    \mu = \partial_t u  - \div \cA (x,t,\nabla u)
\end{equation*}
be the associated Riesz measure. Let $k\in\mathbb R$. Then, $u_k=\min(u,k)$ is a supersolution to \eqref{eq_main}. Additionally, let
\begin{equation}\label{eq:nov10-1}
    \mu_k=\partial_t u_k -\div \mathcal A(x,t,\nabla u_k)
\end{equation}
be the Riesz measure of $u_k$. Then, for any $\varphi\in \Cz^\infty_0(\Omega_T)$, it holds
\begin{equation}
    \label{eq: ineq for measure of cutoff supersol}
    \int_{\Omega_T} \chi_{\{u<k\}}\varphi\dmu \leq \int_{\Omega_T} \varphi \d\mu_k\,.
\end{equation}

\end{lemma}
\begin{proof} 
    In view of the limited regularity of supersolutions to \eqref{eq_main}, we shall employ Steklov averages in the proof. For a given function $w \in L^p_{\loc}(0, T, W^{1, p}_{\loc}(\Omega))$ and $h \in (0, T)$, we define
\[
w^h(x,t) \coloneqq \begin{cases}\dashint_t^{t+h}w(x,\tau)\,d\tau\, \quad &\text{for } (x,t)\in\Omega\times(0,T-h)\,,\\
0 \, \quad &\text{for } (x,t)\in\Omega\times[T-h,T)\,.
\end{cases}
\]
In what follows, we shall repetitively use the fact that
\begin{equation}\label{eq:conv}
    \text{if $w \in L^p_{\loc}(0, T; W^{1, p}_{\loc}(\Omega))$, then $w^h \xrightarrow{h \to 0} w$ in $L^p_{\loc}(0, T; W^{1, p}_{\loc}(\Omega))$\,.}
\end{equation}
We also observe that for a.e.~$(x, t) \in \Omega \times (0, T-h)$, we have
\begin{equation}\label{eq:oct20-1}
\nabla u^h(x, t)=[\nabla u]^h(x, t) \quad  \text{and} \quad \partial_t u^h(x, t)=\frac{u(x,t+h)-u(x,t)}{h}\,.
\end{equation}
By $[\cA(\cdot,\cdot,\nabla u)]^h$, we shall denote the Steklov average of the function $(x, t) \mapsto \cA(x,t,\nabla u(x, t))$.
Let $h \in (0, T/2)$ and $\phi\in C_c^{\infty}(\Omega \times (h, T-h))$. Integrating by parts and using the definition of $\mu$, gives
\begin{equation}
\label{eq:steklov-ineq}
\begin{split}
&\int_{\Omega_T} -u^h \,\partial_t\phi + [\cA(x,t,\nabla u)]^h\cdot\nabla\phi \dx\dt \\
    &= \int_{\Omega} \int_{h}^{T-h} -u^h \,\partial_t\phi + [\cA(x,t,\nabla u)]^h\cdot \partial_t \left(\int_0^t \nabla\phi(x,s) \,\ds\right) \dt\dx \\
    &= \int_{\Omega} \int_{h}^{T-h} \frac{u(x, t+h) - u(x, t)}{h}\phi\dx\dt \\
    & \quad - \int_{\Omega} \int_{h}^{T-h}\frac{\cA(x, t+h, \nabla u(x, t+h)) - \cA(x, t, \nabla u(x, t))}{h} \left( \int_{0}^{t} \nabla \phi(x, s)\ds\right) \dt\dx \\
    &= \int_{\Omega} \int_{h}^{T} -u \frac{\phi(x, t) - \phi(x, t-h)}{h} + \cA(x, t, \nabla u) \left( \dashint_{t-h}^{t} \nabla \phi(x, s)\ds \right) \dt\dx \\
    & = \int_{\Omega_{h,T}} \dashint_{t-h}^t \phi(x, s) \ds\d\mu\,.    
\end{split}
\end{equation}
For every $\ell\in\N$, we denote
\[
v_\ell(s)=\min(\ell (k-s)_+,\,1), \qquad s\in\R\,.
\]
Note that $0\le v_\ell\le1$ and $v'_\ell(s)=-\ell\,\chi_{\{k-1/\ell<s<k\}}$ a.e.~in $\R$. Let $h_0 \in (0, T/2)$, and let $\varphi\in C_c^\infty(\Omega \times (h_0, T-h_0))$ be non-negative. For $h \in (0, h_0),$ we introduce the function
\[
\Psi_{h,\ell} \coloneqq \varphi\, v_\ell(u^h)\,.
\]
Since $v_\ell$ is Lipschitz, $u^h$ is absolutely continuous with respect to~$t$, and $\nabla u^h(\cdot, t)\in L^p_{\loc}(\Omega)$ for every $t$, we have
\begin{equation}\label{eq:oct20-2}
\partial_t v_\ell(u^h)=v_\ell'(u^h)\,\partial_t u^h, 
\qquad \nabla v_\ell(u^h)=v_\ell'(u^h)\,[\nabla u]^h.
\end{equation}
Further, since $h < h_0$ and $u$ is continuous, we conclude that $u^h$ is bounded on the support of $\varphi$. From~\eqref{eq:oct20-1} and~\eqref{eq:oct20-2}, we conclude that $\Psi_{h, \ell}$ is Lipschitz continuous with respect to $t$ and $\Psi_{h, \ell} \in L^p(0, T;W^{1, p}_{\loc}(\Omega))$.
Note also that $\Psi_{h, \ell} \in C_c(\Omega \times (h_0, T-h_0))$. Thus, approximating $\Psi_{h,\ell}$ uniformly, and $\partial_t \Psi_{h, \ell}$ and $\nabla \Psi_{h, \ell}$ in $L^p(\Omega_T)$, by smooth functions, \eqref{eq:steklov-ineq} allows us to write
\begin{equation}\label{eq:main-ineq}
	\int_{\Omega_T}-u^h\partial_t\Psi_{h, \ell}
	+ [\cA(\cdot,\cdot,\nabla u)]^h\cdot\nabla \Psi_{h, \ell}\dx\dt
	\ = \int_{\Omega_{h,T}} \dashint_{t-h}^t \Psi_{h, \ell}(x, s)\ds\d\mu\,.
\end{equation}
Now, we pass to the limit $h \to 0$ in the terms of \eqref{eq:main-ineq} separately.
As for the right-hand side of~\eqref{eq:main-ineq}, we observe that by continuity of $u$ and compactness of $\supp \varphi$, we have that
\begin{equation*}
   \dashint_{t-h}^{t} \Psi_{h, \ell}\dx \xrightarrow{h \to 0} \vp v_\ell(u) \quad \text{uniformly in $(x, t)\,$.}
\end{equation*}
Therefore, as $\mu$ is finite on $\supp \vp$, and $\supp \Psi_{h, \ell} \subset \Omega \times (h_0,T-h_0)$, we conclude
\begin{equation}\label{eq:sep29-1}
    \lim_{h \to 0} \int_{\Omega_{h,T}} \dashint_{t-h}^t \Psi_{h, \ell}(x, s)\ds\d\mu
    = \lim_{h \to 0} \int_{\Omega_T} \dashint_{t-h}^t \Psi_{h, \ell}(x, s)\ds\d\mu
    = \int_{\Omega_T} \varphi v_{\ell}(u)\d\mu\,.
\end{equation}
By expanding $\nabla\Psi_{h, \ell}$ using the product rule, we may write the left-hand side of~\eqref{eq:main-ineq} as
\begin{equation}
\label{eq:sep29-5}
\begin{split}
    &\int_{\Omega_T}-u^h\partial_t\Psi_{h, \ell}
	+ [\cA(\cdot,\cdot,\nabla u)]^h\cdot\nabla \Psi_{h, \ell}\dx\dt \\
    &= \int_{\Omega_T} -u^h\,\partial_t \Psi_{h, \ell}\dx\dt + \int_{\Omega_T} [\cA(\cdot,\cdot,\nabla u)]^h\cdot \big(v_\ell(u^h)\nabla\varphi\big)\dx\dt \\
    & \quad + \int_{\Omega_T} [\cA(\cdot,\cdot,\nabla u)]^h\cdot \big(\varphi\nabla v_\ell(u^h)\big)\dx\dt  \\
    &\eqqcolon \mathrm{I}_{h,\ell} + \mathrm{II}_{h,\ell} + \mathrm{III}_{h,\ell}\,.
\end{split}
\end{equation}
We proceed by estimating the aforementioned three terms, starting with $\mathrm{I}_{h,\ell}$. As $\Psi_{h, \ell}$ has compact support and is absolutely continuous with respect to~$t$, it holds that
\begin{equation*}
    \int_{\Omega_T} \partial_t \Psi_{h, \ell}\dx\dt = 0\,.
\end{equation*}
Thus, by summing and subtracting $\int_{\Omega_T} k \partial_t \Psi_{h, \ell}\dx\dt$, we have
\begin{equation}\label{eq:sep30-1}
\begin{split}
    \mathrm{I}_{h,\ell}
	& =\int_{\Omega_T} (k-u^h)\partial_t \Psi_{h, \ell}\dx\dt \\
    & = \int_{\Omega_T} (k-u^h)\,\varphi\,v'_\ell(u^h)\,\partial_t u^h\dx\dt + \int_{\Omega_T} (k-u^h)v_\ell(u^h)\,\partial_t \varphi\dx\dt\,.
\end{split}
\end{equation}
On the support of $v'_\ell(u^h)$, one has $k-u^h=\tfrac1\ell v_\ell(u^h)$, which gives us that
\begin{equation}
\label{eq:sep30-2}
\begin{split}
    \int_{\Omega_T} (k-u^h)\,\varphi\,v'_\ell(u^h)\,\partial_t u^h \dx\dt
    &=
    \int_{\Omega_T} \varphi\, v_\ell(u^h)\,\partial_t \left(v_\ell(u^h)\right) \dx\dt
    \\ & = 
    \frac{1}{2\ell}\int_{\Omega_T} \partial_t \big( \varphi\, v_\ell(u^h)^2 \big)\dx\dt
    -\frac{1}{2\ell}\int_{\Omega_T} (\partial_t\varphi)\,v_\ell(u^h)^2\dx\dt
     \\ &= -\frac{1}{2\ell}\int_{\Omega_T} (\partial_t\varphi)\,v_\ell(u^h)^2\dx\dt\,
\end{split}
\end{equation}
where in the last equality, we used the compactness of the support of $\varphi$ and the fact that $\vp v_\ell(u^h)^2$ is absolutely continuous with respect to~$t$. Note that
\begin{equation*}
    |(k-u^h)v_\ell(u^h)\,\partial_t \varphi| \leq \left(k + |u^h|\right)|\partial_t \vp|\,,
\end{equation*}
which is bounded on $\Omega_T$ by continuity of $u$. This fact combined with Dominated Convergence Theorem and equalities~\eqref{eq:sep30-1} and~\eqref{eq:sep30-2} gives
\begin{equation}\label{eq:sep29-2}
\limsup_{h \to \infty} \mathrm{I}_{h, \ell} \leq \int_{\Omega_T} (k-u)v_\ell(u)\,\partial_t \varphi\dx\dt + \frac{\|\partial_t \varphi\|_{L^\infty(\Omega_T)}}{2\ell}\,.
\end{equation}
Next, we estimate $\mathrm{II}_{h,\ell}$. By Assumption \eqref{as:growth}, we have that
\begin{equation*}
    \big| [\cA(\cdot,\cdot,\nabla u)]^h(x, t) \cdot \big(v_\ell(u^h(x, t))\nabla\varphi(x, t)\big) \big| \leq L|\nabla \varphi(x, t)| \dashint_{t}^{t+h} |\nabla u(x, s)|^{p-1}\ds\,.
\end{equation*}
Thus,~\eqref{eq:conv} in conjunction with Vitali Convergence Theorem allows us to conclude that
\begin{equation}\label{eq:sep29-3}
    \mathrm{II}_{h, \ell} \xrightarrow{h \to 0} \int_{\Omega_T} \cA(x,t,\nabla u) \cdot \big(v_\ell(u)\nabla\varphi\big)\dx\dt\,.
\end{equation}
Finally, to estimate $\mathrm{III}_{h,\ell}$, note that
\begin{equation*}
    \big|\,v'_\ell(u^h)\, [\cA(\cdot,\cdot,\nabla u)]^h\cdot[\nabla u]^h \varphi \big|
\leq \ell \left(\big| [\cA(\cdot,\cdot,\nabla u)]^h \big|^{\frac{p}{p-1}}+ \big|[\nabla u]^h \big|^p\right) \varphi\,,
\end{equation*}
and that Assumption \eqref{as:growth} on $\cA$ implies that $|\cA(\cdot, \cdot, \nabla u)|^{\frac{1}{p-1}} \in L^p_{\loc}(\Omega_T)$. Therefore, by~\eqref{eq:conv} and the Reverse Fatou Lemma, we obtain
\begin{equation}\label{eq:sep29-4}
\begin{split}
    \limsup_{h \to 0} \mathrm{III}_{h, \ell} & \leq \int_{\Omega_T} \limsup_{h \to 0}v'_\ell(u^h)\, \varphi \cA(\cdot,\cdot,\nabla u)\cdot \nabla u \dx\dt \\
    & \leq \nu \int_{\Omega_T} \limsup_{h \to 0}v'_\ell(u^h)\, \varphi|\nabla u|^p \dx\dt \leq 0\,,
\end{split}
\end{equation}
where the last two inequalities come from the fact that $v_{\ell}'$ is non-positive, while $\vp$ is non-negative.
Inserting \eqref{eq:sep29-2}--\eqref{eq:sep29-4} into \eqref{eq:sep29-5}, we find that
\begin{equation}
\label{eq:nov19-1}
\begin{split}
    &\limsup_{h \to 0} \int_{\Omega_T}-u^h\partial_t\Psi_{h, \ell} + [\cA(\cdot,\cdot,\nabla u)]^h\cdot\nabla \Psi_{h, \ell}\dx\dt \\ &\leq \int_{\Omega_T} (k-u)v_\ell(u)\,\partial_t \varphi + \cA(x,t,\nabla u) \cdot \big(v_\ell(u)\nabla\varphi\big)\dx\dt + \frac{\|\partial_t \varphi\|_{L^\infty(\Omega_T)}}{2\ell}\,.
    \end{split}
\end{equation}
Now, taking $\limsup$ as $h \to 0$ in \eqref{eq:main-ineq}, by means of \eqref{eq:sep29-1} and \eqref{eq:nov19-1}, we obtain
\begin{equation}\label{eq:sep29-6}
 \int_{\Omega_T} \varphi\,v_\ell(u)\d\mu\leq\int_{\Omega_T} (k-u)v_\ell(u)\,\partial_t \varphi + \cA(x,t,\nabla u) \cdot \big(v_\ell(u)\nabla\varphi\big)\dx\dt + \frac{\|\partial_t \varphi\|_{L^\infty(\Omega_T)}}{2\ell} \,.
\end{equation}
Note that $v_\ell(u) \to \chi_{\{u<k\}}$ a.e.~as $\ell \to \infty$. Thus, passing to the limit $\ell \to \infty$ in~\eqref{eq:sep29-6} gives 
\begin{equation}\label{eq:sep29-7}
    \int_{\Omega_T \cap \{u < k\}} \varphi\d\mu\leq \int_{\Omega_T \cap \{u < k\}} (k-u)\partial_t \varphi + \cA(x,t,\nabla u) \cdot \nabla\varphi\dx\dt\,.
\end{equation}
Observe that by \eqref{as:growth} we have $\cA(t,x,0) = 0$, and therefore 
\begin{equation}\label{eq:sep29-8}
    \chi_{\{u < k\}}\cA(x, t, \nabla u) = \cA(x, t, \nabla \min(u, k))\,.
\end{equation}
Moreover, since $\varphi$ is smooth and has compact support in $\Omega_T$, we have that
\begin{equation}\label{eq:sep29-9}
\begin{split}
    \int_{\Omega_T \cap \{u < k\}} (k-u)\partial_t \varphi\dx\dt &  = \int_{\Omega_T} \left(\chi_{\{u < k\}}(k-u) - k\right)\partial_t \varphi\dx\dt \\
    & = \int_{\Omega_T} -\min(u,k)\partial_t \varphi\dx\dt\,.
\textbf{}    
\end{split}
\end{equation}
By substituting equalities~\eqref{eq:sep29-8} and~\eqref{eq:sep29-9} in~\eqref{eq:sep29-7}, we obtain
\[
0 \leq \int_{\Omega_T} \chi_{\{u<k\}}\varphi\d\mu \leq \int_{\Omega_T} -\min(u,k)\,\partial_t\varphi + \cA(x,t,\nabla\min(u,k))\cdot\nabla\varphi\dx\dt\,.
\]
This implies that $\min(u,k)$ is a weak supersolution to \eqref{eq_main}, and the associated measure $\mu_k$ satisfies
\[
\int_{\Omega_T}\chi_{\{u<k\}}\varphi\dmu \ \le\ \int_{\Omega_T}\varphi\d\mu_k\,.
\]
This concludes the proof of the lemma.
\end{proof}
The next lemma allows estimating the decay of the Riesz measure of the solution to the obstacle problem, using the oscillation of the obstacle.
\begin{lemma}
\label{lemma: decay of measure}
Let $\Omega \subset \mathbb{R}^n$ be open, $K\subset \Omega_T$ be compact, and $\psi$ be continuous such that 
\begin{equation}\label{eq:Holder-cont}
    \lvert \psi(z_1)-\psi(z_2) \rvert \leq M d_\alpha^\alpha(z_1,z_2)
\end{equation}
for some $M>0$, and every $z_1\in K$ and $z_2\in \Omega_T$. Let $u=R^\psi$ and denote
\begin{equation*}
    \mu = \partial_t u -\div \mathcal{A}(x,t,\nabla u)\,.
\end{equation*}
Let $z_0\in K$ and $r>0$ be such that $Q_{2r, (2r)^{\alpha}}(z_0) \Subset \Omega_T$. There exists a constant $c=c(n,p,\nu,L,M)$ such that
\begin{equation*}
    \mu(Q_{r,r^\alpha}(z_0))\leq cr^{n+\alpha}\,.
\end{equation*}
\end{lemma}
\begin{proof}
Let us denote $Q \coloneqq Q_{2r, (2r)^{\alpha}}(z_0)$ and
\begin{equation*}
    k \coloneqq \sup_{Q} \psi + r^{\alpha}, \quad v \coloneqq \min(u,k)-\inf_{Q} \psi\,.
\end{equation*}
Then, $v$ is a positive supersolution and we denote
\begin{equation*}
    \mu_v = \partial_t v - \div \mathcal A(x,t,\nabla v)\,.
\end{equation*}
From~\eqref{eq:Holder-cont}, we infer that
\begin{equation*}
    \sup_{Q} v \leq k - \inf_{Q} \psi = r^{\alpha} + \sup_{Q} \psi - \inf_{Q} \psi \leq Cr^\alpha\,,
\end{equation*} 
which by~\Cref{lemma: general estimate for measure of supersol} gives 
\begin{equation*}
    \mu_v (Q_{r,r^\alpha}(z_0)) \leq Cr^{n+\alpha}\,.
\end{equation*}
Recall that $u$ is continuous, as it is a solution to the obstacle problem. Note also that $\mu(Q\cap \{u \geq k\})=0$, as $u$ is $\cA$-harmonic on $\{u > \psi\}$. Hence, Lemma~\ref{lem:meas-inq} and the preceding estimate imply
\begin{equation*}
\mu(Q_{r, r^{\alpha}}(z_0))=\mu(Q_{r, r^{\alpha}} \cap \{u < k\}) \leq \mu_v(Q_{r, r^{\alpha}}(z_0))\leq Cr^{n+\alpha}\,,
\end{equation*}
which completes the proof.
\end{proof}

At this stage, we are ready to prove Theorem~\ref{thm:main}.
\begin{proof}[Proof of Theorem \ref{thm:main}]

Fix a cylinder $Q \Subset \Omega_T$ with smooth lateral boundary such that $E \cap Q \neq \emptyset$, and let $K \subset \big(E \cap Q \big)$ be compact.
By Lemma \ref{lem:obstacle-existence}, there exists a solution to the obstacle problem $v = R^u \big( Q \big)$ in the sense of Definition~\ref{def:obstacle-problem}.
Let $\mu$ be the measure associated  with $v$, that is,
$$
    \mu = \partial_t v - \div \cA(x,t,\nabla v)\,.
$$
Following the arguments in \cite[Section 6]{kuusi2014sharp}, we first prove that
\begin{equation}
    \supp(\mu) \subset (E \cap Q)\,.
    \label{eq:support-measure}
\end{equation}
To this end, consider a cylinder
\begin{equation*}
\widetilde{Q} \coloneqq B(x_0,\rho) \times (t_1,t_2) \subset (Q \setminus E)\,.
\end{equation*}
Since $v \in \Cz(\overline{Q})$, by Lemma \ref{lem:existence-Cauchy-Dirichlet}, there exists a unique weak solution $\tilde{v} \in \Cz(\widetilde{Q})$ to \eqref{eq_main} in $\widetilde{Q}$, in the sense of Definition \ref{weak_sol}, with boundary values $\tilde{v}=v$ on $\partial_\mathrm{par}\widetilde{Q}$.
On the one hand, since $v$ is a weak supersolution to \eqref{eq_main} in $\widetilde{Q}$, by the standard comparison principle in Lemma \ref{lem:standard-comparison}, we obtain that $v \geq \tilde{v}$ in $\widetilde{Q}$.
On the other hand, the following arguments show that $\tilde{v} \geq v$ in $\widetilde{Q}$:
Since $v$ solves the obstacle problem with obstacle function $u$, we have that $u \leq v = \tilde{v}$ on $\partial_\mathrm{par}\widetilde{Q}$.
Hence, applying Lemma \ref{lem:standard-comparison} to the weak solutions $u,\tilde{v}$ to \eqref{eq_main} in $\widetilde{Q}$, we find that $u \leq \tilde{v}$ in $\widetilde{Q}$.
Therefore, setting
\begin{equation*}
    F = \{z \in Q : u(z) = v(z)\}\,,    
\end{equation*}
we obtain that $v \leq \tilde{v}$ on $\left(\partial_\mathrm{par}\widetilde{Q}\right) \cup F$.
Observe that $v$ is a weak solution to \eqref{eq_main} in $\widetilde{Q} \setminus F$. Thus, applying the comparison principle in Lemma \ref{lem:elliptic-comparison} with $U= \widetilde{Q} \setminus F$ leads us to $v \leq \tilde{v}$ in $\widetilde{Q} \setminus F$, and thus to $v \leq \tilde{v}$ in $\widetilde{Q}$.
Altogether, at this point we conclude that $v = \tilde{v}$ in $\widetilde{Q}$, which in particular shows that $v$ is a weak solution to \eqref{eq_main} in $\widetilde{Q}$.
Since $\widetilde{Q} \subset (Q \setminus E)$ is arbitrary, this implies that $v$ is a weak solution to \eqref{eq_main} in $Q \setminus E$. This finishes the proof of \eqref{eq:support-measure}.

Now, we show that $\mu \equiv 0$ in $E \cap Q$.
To this end, since we assume that \eqref{eq:Holder-condition-E} holds true, applying Lemma \ref{lemma: decay of measure} yields
$$
    \mu \big( Q_{r,r^\alpha}(z_0) \big) \leq c r^{n+\alpha}
$$
for any $z_0 \in K$ and any radius $0 < r < r_0 = \min \left\{\frac{1}{2} \dist(K,\partial Q), \dist(K,\partial Q)^\frac{1}{\alpha} \right\}$.
Since $\mathcal{H}^{n+\alpha}_{(\alpha)}(K) \leq \mathcal{H}^{n+\alpha}_{(\alpha)}(E) = 0$, for any $\varepsilon>0$, there exists a family of cylinders $Q_{r_i,r_i^\alpha}(z_i)$, $i \in \N$, such that $x_i \in K$ and $0 < r_i < r_0$ for any $i \in \N$, and such that
$$
    \sum_{i=1}^\infty r_i^{n+\alpha} < \varepsilon\,.
$$
Combining the preceding two observations, we obtain that
$$
    \mu(K) \leq
    \sum_{i=1}^\infty \mu \big( Q_{r_i,r_i^\alpha}(z_i) \big)
    \leq
    c
    \sum_{i=1}^\infty r_i^{n+\alpha}
    < c \varepsilon\,.
$$
Since $\varepsilon$ is arbitrary, this proves that $\mu(K) = 0$. Since this holds for any compact set $K \subset (E \cap Q)$, we have that $\mu(E \cap Q) = 0$ due to inner regularity of $\mu$.
In turn, combining this with \eqref{eq:support-measure}, we conclude that $\mu \equiv 0$ in $Q$, which implies that $v$ solves \eqref{eq_main} in $Q$, see Section \ref{sec:supersolutions-measures}.

Next, we consider $w = -R^{-u}_{\widetilde{\cA}}$, where the operator $\widetilde{\cA} \colon \Omega_T \times \R^n \to \R^n$ is defined by $\widetilde{\cA}(x,t,\xi) = - \cA(x,t,-\xi)$.
Since $\widetilde{\cA}$ satisfies the structure conditions~\eqref{as:growth} and~\eqref{as:monotonicity-weaker}, existence is due to Lemma \ref{lem:obstacle-existence}.
Further, since \eqref{eq:Holder-condition-E} holds for $u$ replaced by $-u$, Lemma \ref{lemma: decay of measure} is applicable to $w$.
Therefore, by the same arguments as for $v$, we conclude that $w$ solves \eqref{eq_main} in $Q$.

At this stage, note that due to the comparison principle in Lemma \ref{lem:standard-comparison}, weak solutions to \eqref{eq_main} in $Q$ are unique.
Therefore, since $v=u=w$ on the parabolic boundary of $Q$, we conclude that $v=u=w$ in $Q$.
In particular, this implies that $u$ is a weak solution to \eqref{eq_main} in $Q$.
Since $Q$ is arbitrary, this concludes the proof of the theorem.
\end{proof}

\begin{remark}
    For simplicity of exposition, we only consider a modulus of continuity $\omega(r)=r^\alpha$. However, our method generalizes in a straightforward way to cover the more general class of $\omega$ as in \cite{kuusi2014sharp}, and gives the corresponding result.
\end{remark}

\section{Proof of Theorem \ref{thm: necessity of Hausdorff zero}}\label{sec:necessity}
Throughout this section, we shall assume that $\cA$ satisfies~\eqref{as:growth} and~\eqref{as:monotonicity-stronger}. Note, however, that the only place where the strict monotonicity is required is in \Cref{lemma: comparison lemma Mingione Kuusi}.

In the statement of \Cref{thm: necessity of Hausdorff zero}, the number $0<\kappa\leq 1$ is related to the Hölder regularity of weak solutions to \eqref{eq_main} in the following sense. We suppose there exists a constant $ \Tilde C>0$ such that for any $R,\omega, \lambda > 0$ and any weak solution $v$ to~\eqref{eq_main} in $Q_{R,\lambda R^\kappa} = Q_{R,\lambda R^\kappa}(z_0)$, it holds that if
\begin{equation}
    \label{eq: intrinsic param in kappa hölder cont} \lambda \geq \Tilde C  \omega/R^\kappa \,\,\,\text{ and } \,\,\, \osc_{Q_{R,\lambda R^\kappa}(z_0)}v \leq \omega\,, 
\end{equation}
then for all $0<r<R$ we have
\begin{equation}
    \label{eq: kappa hölder cont of sols}
    \osc_{Q_{r,\lambda r^\kappa}(z_0)}v  \leq \lambda r^\kappa\,.
\end{equation}
For $p=2$, the above amounts to 
\begin{equation*}
    \osc_{Q_{r}(x_0,t_0)} v \leq c \left(\frac{r}{R}\right)^\kappa \osc_{Q_{R}(x_0,t_0)} v\,, 
\end{equation*}
where $Q_r(x_0,t_0) = B(x_0,r) \times (t_0-r^2,t_0+r^2)$. However, for $p>2$, the intrinsic formulation is more appropriate. In the following proposition, we show that there exists $0<\kappa\leq 1$ and $\Tilde C>1$, that can be determined only in terms of the data, such that \eqref{eq: kappa hölder cont of sols} holds. We note, however, that in the rest of the article, we consider any pair  $(\kappa,\Tilde C)$ such that \eqref{eq: kappa hölder cont of sols} holds.
\begin{proposition}
\label{prop: Hölder regularity for solutions}
There exists $\Tilde C>1$ and $0<\kappa\leq 1$, depending only on $n$, $p$, $\nu$, and $L$, such that \eqref{eq: kappa hölder cont of sols} holds. 
\end{proposition}
\begin{proof}
Let $R,\omega, \lambda > 0$ and let $u$ be a weak solution to~\eqref{eq_main} in $Q_{R,\lambda R^\kappa}(z_0)$ such that \eqref{eq: intrinsic param in kappa hölder cont} holds. The constants $\Tilde C$ and $\kappa>0$ will be fixed during the proof. Denote $z_0=(x_0,t_0)$. Let $C_1$ be as in \Cref{lemma: expansion of pos} and let $\Tilde R = (4 C_1 16^{p-2})^\frac{-1}{p+(2-p)\kappa} R$. Then, $0<\Tilde R<R$ and it is enough to show that \eqref{eq: kappa hölder cont of sols} holds for $0<r<\Tilde R$. For $k\in\mathbb N$, let $r_k=4^{-k}\Tilde R$ and
\begin{equation*}
    Q_k= B(x_0,r_k)\times\Big(t_0-\frac{1}{4}C_1 16^{p-2}\lambda^{2-p}r_k^{p+(2-p)\kappa}, t_0+\frac{1}{4}C_1 16^{p-2}\lambda^{2-p}r_k^{p+(2-p)\kappa}\Big).
\end{equation*}
By the definition of $\Tilde R$, we have $Q_0\subset Q_{R,\lambda R^\kappa}(z_0)$. We claim that
\begin{equation}
    \label{eq: oscillation reduction induction}
    \osc_{Q_k} u \leq \frac{1}{4}\lambda r_{k}^\kappa\,, 
\end{equation}
for all $k\in\mathbb N$. For $k=0$, we have
\begin{equation*}
    \osc_{Q_0} u \leq \osc_{Q_{R,\lambda R^\kappa}(z_0)} u \leq \omega \leq c \Tilde R^\kappa \frac{\omega}{R^\kappa}\leq \frac{1}{4}\Tilde R^\kappa \lambda\,, 
\end{equation*}
provided that the constant $\Tilde C$ is large enough, depending only on $(n,p,\nu,L)$. We proceed by induction. Suppose that \eqref{eq: oscillation reduction induction} holds for $k\in\mathbb N$. Since there is nothing left to prove if 
\begin{equation*}
    \osc_{Q_k} u \leq \frac{1}{4}\lambda r_{k+1}^\kappa\,,
\end{equation*}
we may assume that 
\begin{equation}
    \label{eq: bound from below in induction}
    \osc_{Q_k} u \geq \frac{1}{4}\lambda r_{k+1}^\kappa\,.
\end{equation}
Let 
\begin{equation*}
    z_k=(x_0,t_k)= \Big(x_0,t_0-\frac{3}{4} C_1 16^{p-2}\lambda^{2-p}r_{k+1}^{p+(2-p)\kappa} \Big)\,,
\end{equation*}
and note that $z_k\in Q_k$. We must have either 
\begin{equation}
    \label{eq: measure density}
    \Big\lvert \Big\{x\in B(x_0,r_{k+1}): \sup_{Q_k}u-u(x,t_k) \geq \frac{1}{4}\osc_{Q_k} u\Big\} \Big\rvert \geq \frac{1}{2} \left\lvert B(x_0,r_{k+1})\right\rvert 
\end{equation}
or
\begin{equation*}
    \Big\lvert \Big\{x\in B(x_0,r_{k+1}): u(x,t_k)-\inf_{Q_k}u \geq \frac{1}{4}\osc_{Q_k} u\Big\} \Big\rvert \geq \frac{1}{2} \left\lvert B(x_0,r_{k+1})\right\rvert\,. 
\end{equation*}
Replacing $u$ with $-u$ if necessary, we can assume that \eqref{eq: measure density} holds. Then, by \eqref{eq: bound from below in induction} we have
\begin{equation*}
    \Big\lvert \Big\{x\in B(x_0,r_{k+1}): \sup_{Q_k}u-u(x,t_k) \geq \frac{1}{16}\lambda r_{k+1}^\kappa\Big\} \Big\rvert \geq \frac{1}{2} \lvert B(x_0,r_{k+1})\rvert\,. 
\end{equation*}
Applying \Cref{lemma: expansion of pos} to the supersolution $\sup_{Q_k}u-u$ to \eqref{eq_main} and recalling the construction of $Q_{k+1}$ and $z_k$ we obtain
\begin{equation*}
    \sup_{Q_k} u - \sup_{Q_{k+1}}u = \inf_{Q_{k+1}} \Big(\sup_{Q_k} u-u \Big) \geq \frac{1}{4}\eta\lambda r_{k}^\kappa\,,
\end{equation*}
where $\eta=\eta(n,p,\nu,L) \in (0,1)$. Rearranging the above and subtracting $\inf_{Q_{k+1}}u$ gives
\begin{equation*}
    \osc_{Q_{k+1}} u \leq \sup_{Q_k} u -\inf_{Q_{k+1}}u -\frac{1}{4}\eta \lambda r_{k}^\kappa\, \leq \osc_{Q_k} u -\frac{1}{4}\eta \lambda r_{k}^\kappa\,.
\end{equation*}
The induction assumption then implies that
\begin{equation*}
     \osc_{Q_{k+1}} u \leq (1-\eta)\frac{1}{4}\lambda r_{k}^\kappa\,.
\end{equation*} 
If $\kappa\leq -\log (1-\eta) / \log 4$, we have that $(1-\eta) \leq 4^{-\kappa}$. The previous inequality then becomes
\begin{equation*}
    \osc_{Q_{k+1}}u\leq \frac{1}{4} 4^{-\kappa}\lambda r_{k}^\kappa=\frac{1}{4}\lambda r_{k+1}^\kappa\,,
\end{equation*}
which shows \eqref{eq: oscillation reduction induction}.
Now, let $k \in \N$ such that $r_k \leq r < r_{k-1}$.
Assuming without loss of generality that $C_1 \geq 4$, we have that
\begin{equation*}
    Q_{r,\lambda r^\kappa}(z_0) \subset Q_{k-1}\,.
\end{equation*} 
Then, \eqref{eq: oscillation reduction induction} and the definition of $r_k$ yields
\begin{equation*}
    \osc_{Q_{r,\lambda r^\kappa}(z_0)} u
    \leq
    \osc_{Q_{k-1}} u \leq \frac{1}{4}\lambda r_{k-1}^\kappa \leq \lambda r^\kappa\,,
\end{equation*}
which finishes the proof.
\end{proof}
A key ingredient in the proof of \Cref{thm: necessity of Hausdorff zero} is a Frostman Lemma.
We first recall the following statement for general compact metric spaces, see \cite[Theorem 8.17]{mattila1999geometry}.
\begin{lemma}
\label{thm:metric-frostman}
Let $(X,d)$ be a compact metric space and let $\sigma\geq 0$. Define 
\begin{equation*}
    \mathcal{H}^\sigma(E)=\lim_{\delta\to 0}\inf \left\{
    \sum_{i=1}^\infty \diam(E_i)^\sigma :
    E \subset \bigcup_{i=1}^\infty E_i,
    \diam(E_i)<\delta
    \right\}
\end{equation*}
for any $E\subset X$. If $\mathcal{H}^\sigma(X)>0$, then there exists a Radon measure $\mu$ on $X$ such that $\mu(X)>0$ and
$$
    \mu(A) \leq \diam(A)^\sigma
    \qquad \text{for all Borel $A \subset X\,.$}
$$
\end{lemma}
The previous lemma implies the following result in our parabolic setting.
\begin{corollary}
\label{cor:parabolic-frostman}
Let $\sigma \geq 0$ and $\alpha > 0$. Let $E\Subset\Omega_T$ be a Borel set such that $\mathcal{H}^{\sigma}_{(\alpha)}(E)>0$.
Then there exists a Radon measure $\mu\not\equiv 0$ with $\supp \mu \Subset E$, and such that $\mu(Q_{r,r^\alpha}(z_0)) \leq (2r)^{\sigma}$ for all $z_0 \in\Omega_T$ and $r>0$.
\end{corollary}
\begin{proof}
Since we know that $\mathcal{H}^{\sigma}_{(\alpha)}(E)>0$, there exists a compact set $K \subset E$ such that $\mathcal{H}^\sigma_{(\alpha)}(K)>0$.
Applying Theorem \ref{thm:metric-frostman} to the compact metric space $(K,d_\alpha)$, there exists a Radon measure $\tilde{\mu}$ on $K$ such that $\tilde{\mu}(K) > 0$ and $\tilde{\mu}(A) \leq \diam(A)^\sigma$ for all $A \subset K$.
For $A \subset \Omega_T$, we set $\mu(A) = \tilde{\mu}(A \cap K)$.
Thus, $\mu \not\equiv 0$ is a Radon measure on $\Omega_T$ with $\supp \mu \Subset E$ and
$$
    \mu(Q_{r,r^\alpha}(z_0))
    \leq
    \diam(Q_{r,r^\alpha}(z_0) \cap K)^\sigma
    \leq
    \diam(Q_{r,r^\alpha}(z_0))^\sigma
    =
    (2r)^\sigma.
$$
This concludes the proof of the corollary.
\end{proof}
From \cite[Lemma 2.10]{KuusiMingione2014}, we infer the following comparison lemma for weak solutions to \eqref{eq: measure RHS}; see also \cite[Lemma 4.1]{KuusiMingione2014Wolff}. In \cite{KuusiMingione2014}, the result is stated for a more restricted class of operators and for $\mu\in L^1(\R^{n+1})$. However, as commented before \cite[Lemma 2.8]{KuusiMingione2014}, the same proof works for a more general class of equations. The crucial assumption is the strict monotonicity of the operator, i.e., \eqref{as:monotonicity-stronger}. 
\begin{lemma}
\label{lemma: comparison lemma Mingione Kuusi}
Let $\mu$ be a Radon measure,  $u\in L^p_{\text{loc}}(0,T;W_{\text{loc}}^{1,p}(\Omega))\cap C(\Omega_T)$ be a weak solution to \eqref{eq: measure RHS} in $\Omega_T$, and $z_0\in\Omega_T$, $r>0$, and $\lambda>0$ be such that
\begin{equation*}
    Q=B(x_0,r)\times(t_0-\lambda^{2-p}r^2,t_0)\Subset\Omega_T\,.
\end{equation*}
Let $v\in C(\Omega_T)$ be a weak solution to \eqref{eq_main} in $Q$ with $v=u$ on $\partial_{\text{par}}Q$.
Then, there exists a constant $c=c(n,p,\nu,L)$ such that 
\begin{equation}
    \label{eq: comparison estimate}
    \left(\dashint_{Q} \lvert u-v\rvert^{p-1}\dx\dt \right)^\frac{1}{p-1} \leq c \lambda r \left( \frac{\mu(Q)}{\lambda r^{n+1}} \right)^\frac{n+p}{(p-1)n+p}\,. 
\end{equation}
\end{lemma}
The above lemma is written in the intrinsic geometry of the gradient. However, the parameter $\lambda>0$ is arbitrary in the statement. As such, writing \eqref{eq: comparison estimate} for $\lambda=\Tilde{\lambda} r^{\alpha-1}$ gives an estimate of the form
\begin{equation}
    \label{eq: comparison estimate on Hölder reg cylinders}
    \left(\dashint_{Q_{r,\lambda r^\alpha}(z_0)} \lvert u-v\rvert^{p-1}\dx\dt \right)^\frac{1}{p-1} \leq c \lambda r^{\alpha} \left( \frac{\mu(Q_{r,\lambda r^\alpha}(z_0))}{\lambda r^{n+\alpha}} \right)^\frac{n+p}{(p-1)n+p}\,. 
\end{equation}
Using this form of the comparison estimate, we prove the following local Hölder continuity result for solutions of measure data problems, provided the measure satisfies a certain decay condition.
\begin{theorem}
\label{thm: measure condition implies hölder cont}
Let $\kappa$ and $\Tilde C$ be as in \eqref{eq: kappa hölder cont of sols} and let $0<\alpha<\kappa$.
For a Radon measure $\mu$, let $u$ be a superparabolic function solving \eqref{eq: measure RHS} in the sense of Definition \ref{def:solution-measure-data}.
Suppose that there exists $M>0$ such that for any $Q_{r,r^\alpha}(z_0)\Subset\Omega_T$, we have
\begin{equation}
    \label{eq: assumption on measure decay}
    \mu(Q_{r,r^\alpha}(z_0))\leq M r^{n+\alpha}\,.
\end{equation}
Then, there exists constants $C\geq 1$ and $0<\delta<1$, depending on $n$, $p$, $\nu$, $L$, $\Tilde C$, $\kappa$, and $\alpha$, such that for any $R>0$ and $\omega>0$, whenever
\begin{equation}
    \label{eq: intrinsic parameter in osc of sol meas data}
    \lambda \geq C \max(M, \omega/R^\alpha,1)\,, 
\end{equation}
and $Q_{R,\lambda R^\alpha}(z_0)\Subset\Omega_T$, with u satisfying the intrinsic relation 
\begin{equation}\label{eq:omegaest}
    \left(\dashint_{Q_{R,\lambda R^\alpha}(z_0)} \lvert u-u_{Q_{R,\lambda R^\alpha}(z_0)}\rvert^{p-1}\dx\dt\right)^\frac{1}{p-1} \leq \omega\,,
\end{equation}
then for all $0<r<\delta R$
\begin{equation}
\label{eq: Hölder cont of meas data sol}
    \osc_{Q_{r,\lambda r^\alpha}(z_0)} u  \leq \lambda r^\alpha\,.
\end{equation}
\end{theorem}

\begin{proof}
Fix $x_0 \in \Omega_T$ and assume that $R>0$, $\omega>0$, and $\lambda>0$ are chosen such that $Q_{R,\lambda R^\alpha}(z_0)\Subset\Omega_T$ and \eqref{eq:omegaest} holds. Let $\tau\in \left(0,\frac{1}{2} \right)$ to be chosen later and denote 
\begin{equation*}
    r_k = \tau^k R\,.
\end{equation*}
We claim that 
\begin{equation}
    \label{eq: induction for mean osc}
    \left(\dashint_{Q_{r_{k},\lambda r_k^\alpha}(z_0)} \lvert u-u_{Q_{r_{k},\lambda {r_k}^\alpha}(z_0)}\rvert^{p-1}\dx\dt\right)^\frac{1}{p-1} \leq \lambda r_k^\alpha
\end{equation}
for all $k\in\N$, where $\lambda$ is as in \eqref{eq: intrinsic parameter in osc of sol meas data} with a constant $C \geq 1$ to be chosen later on. We proceed by induction. By choice of $\lambda$, we have that \eqref{eq: induction for mean osc} holds for $k=0$. Suppose that \eqref{eq: induction for mean osc} holds for $k\in\mathbb N$. We denote
\begin{equation*}
    Q_k \coloneqq Q_{r_{k},\lambda r_k^\alpha}(z_0)\,.
\end{equation*}
We will use the quasiminimality of the mean value several times during the proof. By this we mean that for any $q\geq 1$, measure $\nu$ and $f\in L^q(E,\nu)$, where $0<\nu(E)<\infty$, and for all $a\in\mathbb R$, it holds that
\begin{equation*}
\begin{split}
    \left(\dashint_E \lvert f - f_E\rvert^q\dnu\right)^\frac{1}{q} & \leq \left(\dashint_{E} \lvert f-a\rvert^{q}\d\nu\right)^\frac{1}{q} + \lvert a-f_E\rvert \\
    & \leq \left(\dashint_{E} \lvert f-a\rvert^{q}\d\nu\right)^\frac{1}{q} + \dashint_{E} \lvert  f-a\rvert \dnu  \\ 
    & \leq 2 \left(\dashint_{E} \lvert f-a\rvert^{q}\d\nu\right)^\frac{1}{q}.
\end{split}
\end{equation*}
We want to apply \Cref{lemma: comparison lemma Mingione Kuusi} to the function $u$. However, at this stage we only know that $u \in L^{p-1} \left( 0,T;W^{1,p-1}(\Omega) \right)$, and the existence of the comparison function $v$ is not clear. For this reason, we apply two approximation procedures. First, let
\begin{equation*}
    u_j \coloneqq \min(u,j), \text{ \,\, and \,\,\, } \mu_j \coloneqq \partial_t u_j - \div \mathcal A(x,t,\nabla u_j)\,,
\end{equation*}
for $j\in \N$. Then we have $u_j\in L^p(0,T;W^{1,p}(Q_0))$ and moreover, since the measures $\mu_j$ converge weakly to $\mu$ by \eqref{eq:weak-convergence-measures-truncations}, we may let $j$ be large enough so that 
\begin{equation}
    \label{eq: inequality of measure for approximation}
    \mu_j(\overline{Q_k})\leq 2 \mu(\overline{Q_k})\leq 2 M r_k^{n+\alpha}.
\end{equation}
Here, we used assumption \eqref{eq: assumption on measure decay} in the second inequality. Further, by the quasiminimality of the mean value, the estimate $|\min(a,j)-\min(b,j)| \leq |a-b|$ for any $a,b \in \R$, and by the induction assumption, we find that
\begin{equation}
    \label{eq: mean osc of cutoff from induc assump}
    \begin{split}
    \left(\dashint_{Q_k} \lvert u_j-(u_j)_{Q_k}\rvert^{p-1}\dx\dt\right)^\frac{1}{p-1}
    &\leq
    2 \left(\dashint_{Q_k} \big\lvert u_j- \min\big((u)_{Q_k},j\big) \big\rvert^{p-1}\dx\dt\right)^\frac{1}{p-1}
    \\ &\leq
    2\left(\dashint_{Q_k} \lvert u-(u)_{Q_k}\rvert^{p-1}\dx\dt\right)^\frac{1}{p-1}
    \leq 2 \lambda r_k^\alpha.
    \end{split}
\end{equation}
We now turn to the second approximation, which is to ensure continuity. Since $u_j$ is lower semicontinuous and bounded in $\overline{Q_0}$, there exists an increasing sequence of continuous functions $\psi_{j,\ell}$, $\ell\in\mathbb N$, converging pointwise to $u_j$ as $\ell\to\infty$. Let $u_{j,\ell}=R^{\psi_{j,\ell}}(Q_0)$ be the solution of the obstacle problem in $Q_0$ with obstacle $\psi_{j,\ell}$. Since $\psi_{j,\ell}\leq u_{j}$, and $u_{j,\ell}$ is the smallest supersolution above $\psi_{j,\ell}$ we have 
\begin{equation*}
    \psi_{j,\ell}\leq u_{j,\ell}\leq u_j,
\end{equation*}
which implies that $u_{j,\ell}$ converges pointwise to $u_j$ as $\ell\to\infty$. Since the obstacles form an increasing sequence, we conclude by \cite[Lemma 3.7]{KKS-Obstacle} that the sequence $(u_{j,\ell})_{\ell\in\N}$ is increasing. Let 
\begin{equation*}
    \mu_{j,\ell}= \partial_t u_{j,\ell} - \div \mathcal A(x,t,\nabla u_{j,\ell})
\end{equation*}
in sense of Definition \ref{def:solution-measure-data}.
From the Dominated Convergence Theorem and the convergence of gradients in Lemma \ref{lem:sequences-supersolutions}, we conclude that $\mu_{j,\ell}\to \mu_j$ weakly in $Q_0$, as $\ell\to\infty$. Let $v_{j,\ell}$ be the unique solution of \eqref{eq_main} with $v_{j,\ell}=u_{j,\ell}$ on $\partial_{\text{par}}Q_k$, whose existence follows from Lemma \ref{lem:existence-Cauchy-Dirichlet}. By the comparison principle in Lemma \ref{lem:standard-comparison}, the sequence $v_{j,\ell}$ is increasing and bounded. Hence, the pointwise limit $v_j$ exists and by Lemma \ref{lem:sequences-supersolutions}, it is a weak solution of \eqref{eq_main} in $Q_k$.
Applying \eqref{eq: comparison estimate on Hölder reg cylinders} to $u_{j,\ell}$ and $v_{j,\ell}$ yields 
\begin{equation*}
    \left(\dashint_{Q_k} \lvert u_{j,\ell}-v_{j,\ell}\rvert^{p-1}\dx\dt\right)^\frac{1}{p-1}\leq c \lambda r_k^\alpha \left(\frac{\mu_{j,\ell}(Q_k)}{\lambda r_k^{n+\alpha}}\right)^\frac{n+p}{(p-1)n+p}\,.
\end{equation*}
Letting $\ell\to\infty$ in the above, by the Dominated Convergence Theorem and the weak convergence of the measures, we obtain
\begin{equation}
\label{eq: comparison estimate for cutoff}
\begin{split}
    \left(\dashint_{Q_k} \lvert u_{j}-v_{j}\rvert^{p-1}\dx\dt\right)^\frac{1}{p-1}& \leq c \lambda r_k^\alpha \left(\frac{\mu_{j}(\overline{Q_k})}{\lambda r_k^{n+\alpha}}\right)^\frac{n+p}{(p-1)n+p} \\
    & \leq c \lambda r_k^\alpha \left(\frac{M}{\lambda }\right)^\frac{n+p}{(p-1)n+p} \\
    & \leq c \lambda r_k^\alpha \left(\frac{1}{C}\right)^\frac{n+p}{(p-1)n+p},
\end{split}
\end{equation}
where we used \eqref{eq: inequality of measure for approximation} in the penultimate inequality. We now turn to estimating the mean oscillations of $u_j$. By the quasiminimality of the mean value, we obtain that
\begin{equation}
\label{eq: estimating mean oscillations with comparison function}
\begin{split}
    \Bigg(\dashint_{Q_{k+1}} &\lvert u_j-(u_j)_{Q_{k+1}}\rvert^{p-1}\dx\dt\Bigg)^\frac{1}{p-1}  \leq  2\left(\dashint_{Q_{k+1}} \lvert u_j-(v_j)_{Q_{k+1}}\rvert^{p-1}\dx\dt\right)^\frac{1}{p-1} \\
    & \leq 2\left(\dashint_{Q_{k+1}} \lvert u_j-v_j\rvert^{p-1}\dx\dt\right)^\frac{1}{p-1} + 2\left(\dashint_{Q_{k+1}} \lvert v_j-(v_j)_{Q_{k+1}}\rvert ^{p-1}\dx\dt\right)^\frac{1}{p-1} \\
    & \leq 2\left(\frac{\lvert Q_k\rvert}{\lvert Q_{k+1}\rvert}\right)^\frac{1}{p-1}\left(\dashint_{Q_{k}} \lvert u_j-v_j\rvert^{p-1}\dx\dt\right)^\frac{1}{p-1}+ 2\osc_{Q_{k+1}} v_j\,.
\end{split}
\end{equation}
The first term on the right-most side of \eqref{eq: estimating mean oscillations with comparison function} has already been estimated in \eqref{eq: comparison estimate for cutoff}. Let us estimate the term $\osc_{Q_{k+1}} v_j$ in \eqref{eq: estimating mean oscillations with comparison function}. Applying \Cref{lemma: sup estimate for subsolutions} to the subsolutions $(v_j-(v_j)_{Q_k})_+$ and $((v_j)_{Q_k}-v_j)_+$ to \eqref{eq_main} with operators $\cA$ and $\Tilde{\cA}(x,t,\xi) = -\cA(x,t,-\xi)$, respectively, we obtain
\begin{equation}\label{eq:nov12-1}
\begin{split}
    \osc_{Q_{r_k/2,\lambda (r_k/2)^\alpha}} v_j
    &=
    \osc_{Q_{r_k/2,\lambda (r_k/2)^\alpha}} \left(v_j - (v_j)_{Q_k} \right)
    \\ &\leq
    \max \left( \sup_{Q_{r_k/2,\lambda (r_k/2)^\alpha}} \left(v_j - (v_j)_{Q_k} \right)_+, \sup_{Q_{r_k/2,\lambda (r_k/2)^\alpha}} \left( (v_j)_{Q_k} - v_j \right)_+ \right)
    \\ &\leq
    \frac{c}{(\lambda r_k^\alpha)^{p-2}}\dashint_{Q_k} \lvert v_j-(v_j)_{Q_k}\rvert^{p-1}\dx\dt +c\lambda r_k^\alpha\,.
\end{split}
\end{equation}
By the quasiminimality of the mean value and the triangle inequality, we have that
\begin{equation*}
\begin{split}
    \dashint_{Q_k} \lvert v_j-(v_j)_{Q_k}\rvert^{p-1}\dx\dt
    & \leq 2 \dashint_{Q_k} \lvert v_j-(u_j)_{Q_k}\rvert^{p-1}\dx\dt \\
    & \leq c\dashint_{Q_k} \lvert u_j-(u_j)_{Q_k}\rvert^{p-1}\dx\dt + c\dashint_{Q_k} \lvert u_j-v_j\rvert^{p-1}\dx\dt\,.
\end{split}
\end{equation*}
Substituting the last display into~\eqref{eq:nov12-1}, and applying the estimates~\eqref{eq: mean osc of cutoff from induc assump} and~\eqref{eq: comparison estimate for cutoff} we obtain
\begin{equation}
\label{eq: estimate for setup of intrinsic osc decay}
\begin{split}
    \osc_{Q_{r_k/2,\lambda (r_k/2)^\alpha}} v_j
    &\leq
    \frac{c}{(\lambda r_k^\alpha)^{p-2}}\left( \dashint_{Q_k} \lvert u_j-(u_j)_{Q_k}\rvert^{p-1}\dx\dt + \dashint_{Q_k} \lvert u_j-v_j\rvert^{p-1}\dx\dt\right) +\lambda r_k^\alpha
    \\ &\leq 
    \frac{c}{(\lambda r_k^\alpha)^{p-2}}\left( \lambda r_k^\alpha + \lambda r_k^\alpha \right)^{p-1}+\lambda r_k^\alpha \leq  c \lambda r_k^\alpha\,.
\end{split}
\end{equation}
Denote $\Tilde\omega \coloneqq c\lambda (r_k/2)^\alpha$, $\Tilde \lambda \coloneqq \Tilde C \Tilde \omega / (r_k/2)^\kappa$, where $0<\kappa\leq 1$ and $\Tilde C\geq 1$ are as in \eqref{eq: kappa hölder cont of sols}. Then, using that $\Tilde{C}^{2-p} \leq 1$ to increase the cylinder, \eqref{eq: estimate for setup of intrinsic osc decay} gives
\begin{equation}
    \osc_{ Q_{r_k/2,\Tilde{\lambda} (r_k/2)^\kappa}} v_j =  \osc_{Q_{r_k/2,C\lambda (r_k/2)^\alpha}} v_j \leq  \osc_{Q_{r_k/2,\lambda (r_k/2)^\alpha}} v_j \leq c\lambda r_k^\alpha=\Tilde\omega\,,
\end{equation}
and hence applying \eqref{eq: kappa hölder cont of sols} with the smaller radius $r_{k+1}$, we find that
\begin{equation*}
    \osc_{Q_{r_{k+1},\Tilde\lambda r_{k+1}^\kappa}} v_j \leq \Tilde\lambda r_{k+1}^\kappa\,.
\end{equation*}
Recalling the definition of $r_k,\Tilde\omega$ and $\Tilde\lambda$, we thus have
\begin{equation}
    \label{eq: result from intrinsic osc reduction}
    \osc_{Q_{r_{k+1},c \tau^{\kappa-\alpha} \lambda r_{k+1}^\alpha }} v_j \leq c \tau^{\kappa-\alpha} \lambda r_{k+1}^\alpha,
\end{equation}
where $c=c(n,p,\nu,L)$. Choose $\tau = \tau(n,p,\nu,L,\Tilde C, \kappa, \alpha)$ such that 
\begin{equation*}
    c\tau^{\kappa-\alpha}\leq \frac{1}{4}\,.
\end{equation*}
Then, using that $(c\tau^{\kappa-\alpha})^{2-p} \geq 4^{p-2} \geq 1$ to decrease the cylinder on the left-hand side, \eqref{eq: result from intrinsic osc reduction} implies 
\begin{equation*}
    \osc_{Q_{k+1}} v_j \leq \frac{1}{4}\lambda r_{k+1}^\alpha\,, 
\end{equation*}
which in conjunction with~\eqref{eq: estimating mean oscillations with comparison function} and~\eqref{eq: comparison estimate for cutoff} yields
\begin{equation*}
    \left(\dashint_{Q_{k+1}} \lvert u_j-(u_j)_{Q_{k+1}}\rvert^{p-1}\dx\dt\right)^\frac{1}{p-1} \leq c\left(\frac{\lvert Q_k\rvert}{\lvert Q_{k+1}\rvert}\right)^\frac{1}{p-1} \left(\frac{1}{C}\right)^\frac{n+p}{(p-1)n+p} \lambda r_k^\alpha+ \frac{1}{2}\lambda r_{k+1}^\alpha\,.
\end{equation*}
Letting $j\to \infty$ and choosing the constant $C=C(n,p,\nu,L,\Tilde C, \kappa ,\alpha)$ to be large enough so that  
\begin{equation*}
    c\left(\frac{\lvert Q_k\rvert}{\lvert Q_{k+1}\rvert}\right)^\frac{1}{p-1} \left(\frac{1}{C}\right)^\frac{n+p}{(p-1)n+p} \leq \frac{\tau^\alpha}{2}\,,
\end{equation*}
we obtain 
\begin{equation}
    \left(\dashint_{Q_{k+1}} \lvert u-u_{Q_{k+1}}\rvert^{p-1}\dx\dt\right)^\frac{1}{p-1} \leq \lambda r_{k+1}^\alpha\,.
\end{equation}
We thus conclude that \eqref{eq: induction for mean osc} holds for all $k\in\mathbb N $ by induction. Then, \eqref{eq: Hölder cont of meas data sol} follows by a Campanato-type argument, which we briefly outline for the convenience of the reader. We claim that 
\begin{equation}
    \label{eq: induction for pointwise osc}
    \osc_{Q_k} u  \leq \lambda r_k^\alpha
\end{equation}
for all $k \geq 1$ after making the constant $C$ in \eqref{eq: intrinsic parameter in osc of sol meas data} larger. This will imply that \eqref{eq: Hölder cont of meas data sol} holds with $\delta=\tau$ . We may assume without loss of generality that $z_0$ is a Lebesgue point of $u$. By \eqref{eq: induction for mean osc}, we have
\begin{equation}
\label{eq: campanato type estimation}
\begin{split}
    \lvert u(z_0)-u_{Q_k}\rvert & \leq \sum_{i=k}^\infty \lvert u_{Q_{i+1}}-u_{Q_i}\rvert
     \leq \sum_{i=k}^\infty \frac{\lvert Q_{i}\rvert}{\lvert Q_{i+1}\rvert}\dashint_{Q_i}\lvert u-u_{Q_i}\rvert\dx\dt \\
    & \leq c \sum_{i=k}^\infty \left(\dashint_{Q_i}\lvert u-u_{Q_i}\rvert^{p-1}\dx\dt\right)^\frac{1}{p-1}
     \leq c \lambda  \sum_{i=k}^\infty r_i^\alpha
     \leq c\lambda r_k^\alpha\,.
\end{split}
\end{equation}
Let $y \in Q_k$ be a Lebesgue point. Let $i \geq 1$, and denote
\begin{equation*}
    \widehat Q_i \coloneqq Q_{r_i,\lambda r_i^\alpha}(y)\,, \quad \Tilde Q \coloneqq Q_{R/2,\lambda (R/2)^\alpha}(y)\,.
\end{equation*}
Since $\Tilde{Q} \subset Q_0$, applying the quasiminimality of the mean value shows that
$$
    \left( \dashint_{\Tilde Q} |u - u_{\Tilde Q}|^{p-1}\dx \dt \right)^{\frac{1}{p-1}}
    \leq
    c\left( \dashint_{Q_0} |u - u_{Q_0}|^{p-1}\dx \dt \right)^{\frac{1}{p-1}}
    \leq
    c \omega.
$$
Therefore, \eqref{eq: induction for mean osc} holds after increasing $C$ in \eqref{eq: intrinsic parameter in osc of sol meas data}.
Hence, we may apply~\eqref{eq: induction for mean osc} to cylinders centered at $y$. In particular, by the same reasoning as in~\eqref{eq: campanato type estimation}, we can estimate $\lvert u(y)-u_{\widehat Q_{k-1}}\rvert \leq c\lambda r_k^{\alpha}$. Moreover, noting that $Q_k\subset \widehat Q_{k-1}$, we estimate 
\begin{equation}\label{eq:nov13-2}
\begin{split}
    \lvert u(y)-u_{Q_k}\rvert & \leq \lvert u(y)-u_{\widehat Q_{k-1}}\rvert + \lvert u_{\widehat Q_{k-1}}-u_{Q_k}\rvert \\
    & \leq c\lambda r_k^\alpha + c \left(\dashint_{\widehat Q_{k-1}}\lvert u - u_{\widehat Q_{k-1}}\rvert^{p-1}\dx\dt\right)^\frac{1}{p-1} \leq c \lambda r_k^\alpha\,,
\end{split}
\end{equation}
where in the last inequality we applied~\eqref{eq: induction for mean osc} to the cylinder $\widehat Q_{k-1}(y)$. By combining~\eqref{eq: campanato type estimation} and~\eqref{eq:nov13-2}, we finally obtain
\begin{equation*}
    \lvert u(z_0)-u(y)\rvert \leq \lvert u(z_0)-u_{Q_k}\rvert +\lvert u(y)-u_{Q_k}\rvert \leq c\lambda r_k^\alpha\,,
\end{equation*}
whichs shows \eqref{eq: induction for pointwise osc}, after making the constant $C$ in \eqref{eq: intrinsic parameter in osc of sol meas data} larger again. 
\end{proof}

The proof of \Cref{thm: necessity of Hausdorff zero} follows in a straightforward way from the previous results.
\begin{proof}[Proof of \Cref{thm: necessity of Hausdorff zero}]
Let $\mu$ be a measure as in \Cref{cor:parabolic-frostman}, with $\sigma=n+\alpha$. By Lemma \ref{lemma:exis_measure}, there exists an $\mathcal A$-superparabolic function $u$ in $\Omega_T$ satisfying~\eqref{eq: measure RHS}, which is a weak solution in $\Omega_T\backslash E$, since $\supp \mu\subset E$. By \Cref{thm: measure condition implies hölder cont}, $u$ is locally $\alpha$-Hölder continuous with respect to $d_\alpha$. To see this, let $z_0\in \Omega_T$ be a Lebesgue point of $u^{p-1}$ and let $C$ be as in \eqref{eq: intrinsic parameter in osc of sol meas data}. Let $0<R_0<1$ be such that $Q_{R_0,C R_0^\alpha}(z_0)\Subset\Omega_T$. Let $0<R<R_0$. We have $C\max(1,R^-\alpha)=CR^{-\alpha}$. Moreover, choosing $R$ small enough, since $z_0$ is a Lebesgue point, we have
\begin{equation*}
    \left(\dashint_{Q_{R, C R^\alpha}(z_0)} \lvert u-u_{Q_{R,C R^\alpha}(z_0)}\rvert^{p-1}\dx\dt\right)^\frac{1}{p-1} < 1,
\end{equation*}
and hence \eqref{eq: Hölder cont of meas data sol} together with a simple covering argument implies the local Hölder continuity. Finally, $u$ does not solve \eqref{eq_main} in $\Omega_T$, since $\mu\not \equiv 0$.
\end{proof}

\bigskip
\textbf{Acknowledgements.}
{We thank Iwona Chlebicka and Juha Kinnunen for helpful discussions, and Juha Kinnunen for proposing the problem.
The second and third authors thank the Department of Mathematics, Informatics, and Mechanics of the University of Warsaw for the hospitality during their research visits. The first and fourth authors would like to express sincere gratitude towards Aalto University for the kind hospitality during their research visits. All the authors would also like to thank the Initiative of Excellence at the University of Warsaw.}

\bigskip
\textbf{Funding.}
This research was funded in whole or in part by the Austrian Science Fund (FWF) project \emph{Evolutionary problems in noncylindrical domains}, grant doi 10.55776/J4853. The first author is supported by Polish Ministry of Science and Education grant PN/02/0001/2023. The second author has been supported by the Vilho, Yrjö and Kalle Väisälä Foundation. The fourth author is supported by the National Science Centre (NCN) Grant Sonata Bis 2019/34/E/ST1/00120.

\printbibliography
\end{document}